\documentclass[a4paper,10pt,reqno]{amsart}
\usepackage{amsthm,amsmath,amssymb,amsfonts}
\usepackage[T1]{fontenc}
\usepackage[utf8]{inputenc}
\usepackage{graphicx}
\usepackage{bbold}
\usepackage{xcolor}
\usepackage{caption}
\usepackage{enumitem}
\usepackage{mathtools}
\usepackage{microtype}
\usepackage{stmaryrd}
\usepackage{tikz-cd}
\usepackage{multirow}
\usepackage{float}
\usepackage{etoolbox}
\patchcmd{\subsubsection}{\itshape}{\bfseries}{}{}
\usepackage{hyperref}
\hypersetup{
    colorlinks = true,
    linkbordercolor = {red},
   	linkcolor ={magenta},
	anchorcolor = {pink},
	citecolor =  {magenta},
	filecolor = {blue},
	menucolor = {blue},
	runcolor =  {blue},
	urlcolor = {magenta},
}

\usepackage{quiver}
\usetikzlibrary{nfold}

\newtheorem{theorem}{Theorem}[section]
\newtheorem{lemma}[theorem]{Lemma}
\newtheorem{proposition}[theorem]{Proposition}
\newtheorem{corollary}[theorem]{Corollary}
\newtheorem{claim}[theorem]{Claim}

\theoremstyle{definition}
\newtheorem{definition}[theorem]{Definition}

\theoremstyle{remark}
\newtheorem{remark}[theorem]{Remark}

\newtheorem*{caveat}{Caveat}

\newcommand{\N}{\mathbb{N}}				
\newcommand{\Z}{\mathbb{Z}}				
\newcommand{\Q}{\mathbb{Q}}				

\renewcommand{\phi}{\varphi}
\renewcommand{\epsilon}{\varepsilon}

\newcommand{\cat}{{\sf C}}								
\newcommand{\catt}{{\sf D}}								
\newcommand{\Set}{{\sf Set}}								
\newcommand{\Pos}{{\sf Pos}}								
\newcommand{\Tot}{{\sf Tot}}								
\newcommand{\CHaus}{{\sf CHaus}}							
\newcommand{\Tych}{{\sf Tych}}							
\newcommand{\CMon}{{\sf CMon}}							
\newcommand{\AGrp}{{\sf AbGrp}}							
\newcommand{\COAMon}{\sf Q} 							
\newcommand{\OAGrp}{{\sf OAbGrp}}						
\newcommand{\B}[1]{\mathcal{B}(#1)}						
\newcommand{\Alg}[1]{\sf{Alg}(#1)}							
\newcommand{\Sat}{{\sf Sat}}								
\newcommand{\Abs}{{\sf Abs}}								
\newcommand{\K}{{\sf K}}									
\newcommand{\M}{{\sf M}}								
\newcommand{\DL}{{\sf DL}}								
\newcommand{\BA}{{\sf BA}}								
\newcommand{\CRng}{{\sf CRng}}							
\newcommand{\RegRng}{{\sf RegRng}}						
\newcommand{\RedRng}{{\sf RedRng}}						
\newcommand{\Field}{{\sf Field}}							
\newcommand{\fp}[1]{{#1}_{\omega}}						
\newcommand{\Lfp}{{\sf Lfp}}								
\newcommand{\Lex}{{\sf Lex}}								
\newcommand{\Cat}{{\sf Cat}}								
\newcommand{\CAT}{{\sf CAT}}								

\newcommand{\ar}[1]{\mathcal{#1}}							
\newcommand{\X}{\ar{X}}									
\newcommand{\id}{\mathrm{id}}								
\newcommand{\into}{\hookrightarrow}						
\newcommand{\emb}{\rightarrowtail}							

\newcommand{\op}{{\mathrm{op}}}						
\newcommand{\ort}[1]{{#1}^{\perp}}						
\DeclareMathOperator{\Ext}{Ext}						
\DeclareMathOperator{\EExt}{Ext_{p}}					
\DeclareMathOperator{\Epi}{Epi}						
\newcommand{\colim}{\operatornamewithlimits{colim}}		

\renewcommand{\leq}{\leqslant}

\newcommand{\IMP}{\Rightarrow}						
\newcommand{\IFF}{\Leftrightarrow}						

\newcommand*\cocolon{%
        \nobreak
        \mskip6mu plus1mu
        \mathpunct{}%
        \nonscript
        \mkern-\thinmuskip
        {:}%
        \mskip2mu
        \relax
}

\newcommand{\ES}{\textnormal{ES}}

\newcommand{\SES}{\textnormal{SES}}
\newcommand{\AP}{\textnormal{AP}}
\newcommand{\IPA}{\textnormal{IPA}}
\newcommand{\TM}{\textnormal{TM}}
\newcommand{\SAP}{\textnormal{SAP}}
\newcommand{\EAP}{\textnormal{EAP}}
\newcommand{\TEM}{\textnormal{TEM}}

\title{Beth companions of finitary essentially algebraic theories}

\author{Ivan Di Liberti}
\address{
Department of Philosophy, Linguistics and Theory of Science\newline
University of Gothenburg\newline
Gothenburg, Sweden
}
\email{diliberti.math@gmail.com}
\thanks{Ivan Di Liberti was supported by the Knut and Alice Wallenberg Foundation, through the grant no.~2020.0199.}

\author{Luca Reggio}
\address{
Dipartimento di Matematica {\it Federigo Enriques}\newline 
Universit\`a degli Studi di Milano\newline 
Milan, Italy
}
\email{luca.reggio@unimi.it}

\begin{document}

\begin{abstract}
For categories, being balanced (meaning that every arrow that is both epic and monic is an isomorphism) can be regarded as a strong tameness property which plays an important role, for example, in algebra and logic.

We study the problem of associating a balanced companion category, called a \emph{Beth companion}, to a locally finitely presentable category (equivalently, the category of models of a finitary essentially algebraic theory). We show that, if it exists, the Beth companion is unique and can be described in terms of \emph{saturated} objects. Under additional assumptions, we prove that Beth companions can be computed as orthogonality classes, and admit a syntactic presentation via Gabriel--Ulmer duality. Furthermore, we establish conditions for the transfer of properties from a category to its Beth companion. 

A number of examples of categories that do or do not admit a Beth companion are provided. In particular, we show that the category of commutative rings does not admit a Beth companion, and employ Pierce sheaf representations to describe the Beth companion of other classes of rings.
\end{abstract}

\maketitle

\section{Introduction}

Epimorphisms, i.e., right-cancellative morphisms, need not be surjective.
One of the best-known examples is the inclusion
\[\Z \into \Q\]
of the integers into the rationals, regarded as a morphism in the category $\CRng$ of commutative (unital) rings and ring homomorphisms.
The problem of establishing whether all epimorphisms in a class of algebras are surjective, and more generally of characterising epimorphic images of algebras, has a long tradition. See e.g.\ \cite{Isbell1966,HI1967,Storrer1968,SAC1968,Reid1969} for some earlier work on this topic. In many familiar categories of algebras, monomorphisms are injective, so the question is equivalent to asking whether the category is \emph{balanced}, meaning that every morphism that is both epic and monic is an isomorphism. In particular, the category $\CRng$ is not balanced. It is natural to wonder if the non-surjectivity of epimorphisms in $\CRng$ could be ``repaired'' in a canonical (universal) way.
The present paper addresses a generalisation of this question: given a non-balanced category, is it possible to associate with it a companion category in which balancedness is enforced? 

For instance, consider the category $\Tych$ of Tychonoff spaces and continuous maps. This category is not balanced: if $X$ is a non-compact Tychonoff space, then the inclusion map $X \into \beta X$ into its Stone--\v{C}ech compactification is epi-monic (because $X$ is dense in $\beta X$) but not an isomorphism. If $\CHaus$ denotes the full subcategory of $\Tych$ consisting of compact Hausdorff spaces, the Stone--\v{C}ech compactification yields a left adjoint to the inclusion $\CHaus \into \Tych$, and the unit of the adjunction is componentwise monic. That is, $\CHaus$ is a \emph{mono-reflective} subcategory of~$\Tych$. Furthermore, $\CHaus$ is balanced because every continuous bijection between compact Hausdorff spaces is a homeomorphism. It would seem reasonable to regard $\CHaus$ as a balanced companion of $\Tych$.

As our motivation primarily stems from algebra, we shall develop the theory of balanced companions for locally finitely presentable (lfp) categories. These can be characterised as the categories of models of finitary \emph{essentially algebraic theories}, which generalise algebraic theories by allowing partially defined operations. All (finitary) varieties and quasivarieties of algebras\footnote{A class of algebras is a \emph{quasivariety} if it is closed under isomorphic copies, subalgebras, Cartesian products, and ultraproducts; a \emph{variety} is a quasivariety that is closed under homomorphic images.} are lfp. The categories $\Tych$ and $\CHaus$, though not lfp and hence outside our framework, exemplify what we shall call a \emph{Beth companion} (for an explanation of this terminology, see the paragraph ``Balancedness, implicit operations, and Beth definability'' below). 
The Beth companion $\B{\cat}$ of an lfp category $\cat$ is a full balanced mono-reflective subcategory of $\cat$ in the category of lfp categories (Definition~\ref{def:Beth-com}).

In Section~\ref{s:preliminaries}, we recall the necessary background material. In Section~\ref{s:Beth-saturated} we show that, whenever it exists, the Beth companion of an lfp category $\cat$ is unique and coincides with the full subcategory $\Sat(\cat)$ of \emph{saturated} objects (Theorem~\ref{t:Beth-companion-unique}).\footnote{While the categories $\Tych$ and $\CHaus$ are not locally presentable, note that compact Hausdorff spaces are precisely the saturated objects in~$\Tych$.} The importance of the latter result resides in the fact that it reduces the question of whether an lfp category admits a Beth companion to an axiomatisability problem (see Remark~\ref{rmk:axiomatisability-Beth-comp}). By replacing saturated objects with \emph{absolutely closed} ones, we obtain in Section~\ref{s:strong-Beth-sbs-closed} a description of the \emph{strong Beth companion}, in which balancedness is strengthened to the requirement that every monomorphism be regular. Saturated and absolutely closed objects can be regarded as relaxations of the concept of an absolutely pure module (see Remark~\ref{rm:purity}). 

In Section~\ref{s:ort-classes-as-Beth-companions} we prove that, if $\cat$ satisfies certain amalgamation-type properties, then its Beth companion coincides with the localisation at all the epi-monos (Theorem~\ref{t:Beth-companion-amalgam-orthogonal}). Furthermore, $\cat$ admits a Beth companion if, and only if, $\Sat(\cat)$ is closed in $\cat$ under directed colimits (Corollary~\ref{cor:Beth-TEM-omega-ort}); the latter is a ``finitary axiomatisability'' condition, and holds precisely when $\Sat(\cat)$ is an $\omega$-injectivity class in the sense of~\cite{RAB2022}. The transfer of properties from $\cat$ to its Beth companion is the topic of Section~\ref{s:transfer-properties}. In particular, we identify sufficient conditions ensuring that if $\cat$ is equivalent to a (quasi)variety, then $\B{\cat}$ is too (Theorem~\ref{t:transfer-quasi-variety}). 

In Section~\ref{s:syntactic-viewpoint} we explain how, under appropriate assumptions, the Beth companion is an $\omega$-orthogonality class in $\cat$ (in the sense of, e.g., \cite[Definition~1.35]{ar94book}) and can therefore be described dually --- via Gabriel--Ulmer duality for lfp categories --- as a coinverter at the level of the corresponding theories (Theorem~\ref{t:syntactic-Beth-companion}). Section~\ref{s:examples} collects several examples of lfp categories that do or do not admit a Beth companion. In particular, we prove that the category $\CRng$ of commutative rings does not admit a Beth companion, and employ Pierce sheaf representations to describe the Beth companion of reduced and (von Neumann) regular commutative rings. A related description of the Beth companion for these two classes of rings was given in~\cite{CKM2026rings} using different methods.

Finally, for the benefit of the reader, in Appendix~\ref{s:acronyms} we provide a list of the balancedness- and amalgamation-type properties employed in the paper, along with pointers to the corresponding definitions.

\subsubsection*{Balancedness, implicit operations, and Beth definability}
The question of characterising the algebraic theories (in the sense of Lawvere) such that the corresponding categories of algebras are balanced, explicitly stated in \cite[p.~57]{Lawvere1968} and still open to this day, has led to the investigation of \emph{implicit partial operations}, notably in the work of H\'ebert \cite{Hebert1993,Hebert1998}; for an overview, see~\cite{Reggio2026}. 
For instance, a quasivariety is balanced precisely when certain implicit partial operations are term-definable, which suggests a strong link with \emph{Beth definability} in logic, stating that implicit concepts admit explicit definitions. For algebraic (propositional) logics, the precise sense in which balancedness corresponds to a Beth definability property was identified by N\'emeti for classes of algebras, cf.~\cite[Theorem~5.6.10]{HMT1985}, and extended to a purely logical setting by Hoogland and Blok \cite{Hoogland2000,BH2006}. To wit, an \emph{algebraisable} (or, more generally, \emph{equivalential}) logic satisfies the \emph{infinite Beth property} if, and only if, the corresponding class of algebras is balanced. This motivates the name ``Beth companion'' for the balanced companion categories studied in the present paper.

\subsubsection*{Related work}
Beth companions for classes of algebras, and especially quasivarieties, were introduced independently in~\cite{CKM2026}; see also~\cite{CKM2026bis}. The approach adopted there relies on implicit operations and universal algebra. For quasivarieties with the amalgamation property, the Beth companion of \emph{op.\ cit.} is a strong Beth companion in our sense (see Remark~\ref{rmk:comparison-Beth-companions}). In the absence of amalgamation, the relationship between the two notions of companions remains unclear. 

\subsubsection*{Acknowledgements}
The authors are grateful to  Michel H\'ebert for pointing out the connection between Beth companions and factorisation systems (Remark~\ref{Hebert}).

\section{Preliminaries}\label{s:preliminaries}

\subsection{Locally finitely presentable categories}\label{s:lfp}

We review the basic facts concerning locally finitely presentable categories; for a thorough treatment, see~\cite{ar94book}. Let $\Set$ denote the category of sets and functions. An object $X$ of a locally small category $\cat$ is \emph{finitely presentable} if the associated hom-functor $\cat(X,-)\colon \cat\to\Set$ preserves directed colimits.

\begin{definition}
A locally small category $\cat$ is \emph{locally finitely presentable} (\emph{lfp}, for short) if it is cocomplete and admits a set $\sf G$ of finitely presentable objects such that every object of $\cat$ is a directed colimit of objects from $\sf G$.
\end{definition}

Lfp categories are precisely the categories of models of (possibly multi-sorted) \emph{finitary essentially algebraic theories} \cite[Theorem~3.36]{ar94book}. Roughly speaking, these generalise finitary equational theories by allowing function symbols that are interpreted as partial operations whose domains are specified by equations; see~\S 3.D in op.~cit.\ for a precise definition.
Any Horn class of first-order structures, with morphisms the homomorphisms, forms an lfp category. In particular, every finitary quasivariety of algebras is lfp; finitely presentable objects in the above sense coincide with finitely presentable algebras in the usual sense. An example of an lfp category that is not equivalent to any Horn class of first-order structures is $\Cat$, the category of small categories; see~\cite{Barr1989}.

Recall that an epimorphism $e$ is a \emph{strong epimorphism} if it has the left lifting property with respect to all monomorphisms; that is, for any commutative square
\[\begin{tikzcd}[cells={nodes={inner sep=2pt, outer sep=1pt}}]
\phantom{\cdot} \arrow{r} \arrow{d}[swap]{e} & \phantom{\cdot} \arrow{d}{m} \\
\phantom{\cdot} \arrow{r} \arrow[dashed]{ur}[description]{d} & \phantom{\cdot}
\end{tikzcd}\]
with $m$ monic, there is a unique arrow $d$ making the two triangles commute. \emph{Strong monomorphisms} are defined dually. Every lfp category has a (strong epi, mono) factorisation system \cite[Proposition~1.61]{ar94book}, which coincides with the usual (surjective, injective) factorisation system in any quasivariety of algebras.

Every lfp category is complete \cite[Corollary~1.28]{ar94book}, and the appropriate notion of morphism between lfp categories is that of an \emph{lfp morphism}, namely a functor that preserves limits and directed colimits. This definition of lfp morphism is motivated by Gabriel--Ulmer duality, cf.\ Section~\ref{s:syntactic-viewpoint}. Any lfp morphism has a left adjoint which preserves finitely presentable objects \cite[Theorem~1.66]{ar94book}.

\subsection{Balancedness-type properties}\label{s:definability-props}
Recall that a monomorphism $m$ is an \emph{extremal monomorphism} provided that, for any decomposition $m= f\circ e$, if $e$ is epic then it is an isomorphism. \emph{Extremal epimorphisms} are defined dually. 
It is a simple observation that the following conditions are equivalent for any category $\cat$:
\begin{enumerate}[label=(\arabic*)]
\item $\cat$ is balanced;
\item every monomorphism in $\cat$ is extremal;
\item\label{i:every-epi-is-extremal} every epimorphism in $\cat$ is extremal.
\end{enumerate}

If $\tau$ is an algebraic signature, the extremal epimorphisms in the category $\Alg{\tau}$ of $\tau$-algebras and homomorphisms between them are precisely the surjective homomorphisms. This property is inherited by any full subcategory that is closed in $\Alg{\tau}$ under binary products and subalgebras.
For this reason, in the universal algebra literature, property~\ref{i:every-epi-is-extremal} above is referred to as the \emph{\ES~property} (``Epimorphisms are Surjective'').
We record this notion, along with two important variations:

\begin{definition}\label{d:ES-and-SES}
A category $\cat$ is said to have the
\begin{enumerate}
\item[(\ES)]\label{ES} \emph{ES property} if every epi in $\cat$ is an extremal epi;
\item[(\SES)]\label{SES} \emph{strong ES property} if every mono in~$\cat$ is a regular mono.
\end{enumerate}
\end{definition}

Note that $\SES\IMP\ES$ because the implications
\[
\text{regular mono} \ \IMP \ \text{strong mono} \ \IMP \ \text{extremal mono}
\]
hold in any category $\cat$. Furthermore, every extremal mono is a strong mono if $\cat$ has pushouts. The dual statements hold for regular, strong, and extremal epis. Thus, since lfp categories are complete and cocomplete, in any such category extremal monos coincide with strong monos, and extremal epis coincide with strong epis.

Following Isbell~\cite{Isbell1966}, let us recall the notions of saturated and absolutely closed objects, which will play a key role in the following.

\begin{definition}
Let $A$ be an object of a category $\cat$. We say that $A$ is
\begin{enumerate}[label=(\roman*)]
\item \emph{saturated} if every monomorphism in $\cat$ with domain $A$ is extremal (equivalently, if every epi-mono in $\cat$ with domain $A$ is an isomorphism);
\item \emph{absolutely closed} if every monomorphism in $\cat$ with domain $A$ is regular.
\end{enumerate}
We denote by $\Sat(\cat)$ and $\Abs(\cat)$ the full subcategories of $\cat$ defined by the saturated objects and the absolutely closed objects, respectively.
\end{definition}

Every absolutely closed object is saturated. Moreover, $\cat$ has the $\ES$ property precisely when $\Sat(\cat)=\cat$, and it has the $\SES$ property precisely when $\Abs(\cat)=\cat$. 

\begin{caveat}
In universal algebra, it is customary to consider injective and surjective homomorphisms between algebras of a given class $\K$. If we regard the latter as a category, with arrows the homomorphisms, the notions of ``surjective'' and ``injective'' are external, as they refer to the forgetful functor to the category of sets. If $\K$ is a variety or a quasivariety of algebras, then the injective (respectively, surjective) homomorphisms between $\K$-algebras coincide with the monomorphisms (respectively, regular epimorphisms or, equivalently, extremal epimorphisms) in $\K$. However, this need not be the case for an arbitrary class of algebras. 

Thus, the categorical formulations of the definability-type properties given above may differ from those in the universal algebra literature when $\K$ is not a quasivariety. A similar warning applies to the amalgamation-type properties in Section~\ref{s:preliminaries-amalgamation}.
\end{caveat}

\subsection{Amalgamation-type properties}\label{s:preliminaries-amalgamation}

A \emph{span of morphisms} in a category $\cat$ is a pair of arrows in $\cat$ with a common domain.
 An amalgamation-type property typically states that any span
 \[\begin{tikzcd}
\phantom{\cdot} & \phantom{\cdot} \arrow{l}[swap]{\alpha} \arrow{r}{\beta} & \phantom{\cdot}
\end{tikzcd}\] 
in a certain class can be completed into a commutative square of the form
\begin{equation}\label{eq:generic-amalgam}
\begin{tikzcd}[cells={nodes={inner sep=2pt, outer sep=1pt}}]
\phantom{\cdot} \arrow{r}{\alpha} \arrow{d}[swap]{\beta} & \phantom{\cdot} \arrow[dashed]{d}{\delta} \\
\phantom{\cdot} \arrow[dashed]{r}{\gamma} & \phantom{\cdot}
\end{tikzcd}
\end{equation}
with certain properties.

\begin{definition}\label{d:AP-SAP-IPA}
A category $\cat$ is said to have
\begin{enumerate}
\item[(AP)]\label{AP} the \emph{amalgamation property} if any span of monomorphisms $\alpha,\beta$ can be completed to a commutative square as in eq.~\eqref{eq:generic-amalgam} with $\gamma, \delta$ monic;
\item[(SAP)]\label{SAP} the \emph{strong amalgamation property} if any span of monomorphisms $\alpha,\beta$ can be completed to a commutative square as in eq.~\eqref{eq:generic-amalgam}, with $\gamma, \delta$ monic, that is also a pullback square;
\item[(IPA)]\label{IPA} the \emph{intersection property of amalgamations} if any span of monomorphisms $\alpha,\beta$ that can be completed to a commutative square as in eq.~\eqref{eq:generic-amalgam}, with $\gamma, \delta$ monic, can also be completed to a pullback square of monomorphisms.
\end{enumerate}
\end{definition}

Note that $\SAP\IFF \AP + \IPA$ (that is, a category satisfies $\SAP$ if, and only if, it satisfies both $\AP$ and $\IPA$).

\begin{remark}
Assume that $\cat$ admits pushouts of spans of monomorphisms. Then $\cat$ has AP if, and only if, monomorphisms are stable under pushouts along monomorphisms. Furthermore, $\cat$ has SAP if, and only if, it has AP and every pushout square consisting entirely of monomorphisms is also a pullback square.
\end{remark}

The following result, due to Tholen (cf.\ \cite[Proposition~1 and Theorem~1]{Tholen1982}), shows that $\SES\IFF\IPA$ for a broad class of categories, including all lfp categories:
\begin{proposition}\label{l:IPA-iff-SES}
Let $\cat$ be any category. Then $\IPA\IMP \SES$, and the converse implication holds if $\cat$ admits cokernel pairs of monomorphisms.
\end{proposition}

Furthermore, it follows from \cite[Proposition~2 and Theorem~2]{Tholen1982} that, in any lfp category $\cat$ satisfying \AP, extremal monomorphisms coincide with regular monomorphisms. Therefore, in such a category, $\Sat(\cat)=\Abs(\cat)$ and properties $\ES,\SES,\IPA$ and $\SAP$ are all equivalent.

We refer interested readers to~\cite{KMPT1983} for a wealth of examples of categories satisfying the balancedness- and amalgamation-type properties introduced above.

\subsection{Orthogonality and injectivity classes}

We review the concepts of orthogonality and injectivity, which play an important role in the context of lfp categories.

\begin{definition}\label{d:orthogonal}
Let $f\colon X\to Y$ be an arrow in a category $\cat$. An object $A$ of $\cat$ is \emph{injective with respect to $f$} if, for every arrow $g\colon X\to A$, there exists an arrow $Y\to A$ making the following diagram commute.
\[\begin{tikzcd}
X \arrow{d}[swap]{f} \arrow{r}{g} & A \\
Y \arrow[dashed]{ur} &
\end{tikzcd}\]
If the dashed arrows are unique, we say that $A$ is \emph{orthogonal to $f$}.
If $\ar{W}$ is a collection of arrows of $\cat$, we say that $A$ is \emph{injective with respect to $\ar{W}$} (respectively, \emph{orthogonal to $\ar{W}$}) if it is injective with respect to each member of $\ar{W}$ (respectively, orthogonal to each member of $\ar{W}$).
\end{definition}

\begin{remark}
The uniqueness condition in Definition~\ref{d:orthogonal} is met automatically if $f$ is epic. In this case, $A$ is orthogonal to $f$ just when it is injective with respect to~$f$.
\end{remark}

Given a class of arrows $\ar{W}$ in $\cat$, the full subcategory of $\cat$ defined by those objects that are orthogonal to $\ar{W}$ is denoted by $\ort{\ar{W}}$. 

\begin{definition}
A full subcategory of $\cat$ is a \emph{(small-)orthogonality class} if it is of the form $\ort{\ar{W}}$ for a (small) collection of morphisms $\ar{W}$ of $\cat$.
\end{definition}

Recall that a \emph{reflective subcategory} of a category $\cat$ is a full subcategory $\catt$ of~$\cat$ such that the inclusion functor $R\colon \catt\into\cat$ is right adjoint. The left adjoint $L$ to~$R$ is called \emph{reflector}.
Every full, \emph{replete} (i.e., isomorphism-closed) and reflective subcategory $\catt$ of $\cat$ is an orthogonality class. In fact, $\catt=\ort{\ar{W}}$, where $\ar{W}$ consists of the components of the unit $\eta$ of the adjunction $L\dashv R$; see e.g.\ \cite[Examples~1.33(1)]{ar94book}. 

However, an orthogonality class in $\cat$ may or may not be a reflective subcategory of~$\cat$; this is known as the \emph{orthogonal subcategory problem}.
This problem has a positive answer when the class of arrows under consideration consists of epimorphisms. This is a special case of more general results; see e.g.~\cite{FK1972,AHS2009}.

\begin{proposition}\label{p:localise-epis}
Let $\cat$ be an lfp category, and let $\ar{W}$ be a class of arrows in $\cat$. If each member of $\ar{W}$ is an epimorphism, then $\ort{\ar{W}}$ is a reflective subcategory of $\cat$. 
\end{proposition}

The orthogonal subcategory problem has a positive answer also for \emph{$\omega$-orthogonality classes} in $\cat$, i.e.\ subcategories of the form $\ort{\ar{W}}$, where $\ar{W}$ is a set of arrows between finitely presentable objects of $\cat$. This is a special case of \cite[Theorem~1.39]{ar94book}.

\begin{theorem}\label{orthog-reflective}
Suppose that $\cat$ is an lfp category. Every $\omega$-orthogonality class in $\cat$ is a reflective subcategory closed under directed colimits in $\cat$ and, moreover, it is lfp.
\end{theorem}

\begin{remark}
The converse of Theorem~\ref{orthog-reflective} is false: not every reflective subcategory of an lfp category $\cat$, closed under directed colimits in $\cat$, is an $\omega$-orthogonality class; see~\cite{Volger1979} for a counterexample. Characterisations of $\omega$-orthogonality classes in lfp categories have been provided in~\cite{HAR2001,AS2004}.
\end{remark}

A useful variation of the notion of an $\omega$-orthogonality class is that of an \emph{$\omega$-injectivity class} in a category $\cat$, which is a full subcategory of $\cat$ consisting of all the objects that are injective with respect to a set of arrows between finitely presentable objects of $\cat$. To characterise $\omega$-injectivity classes, recall from \cite[Definition~2.27]{ar94book} that, in any category, one can define the notion of an $\omega$-pure morphism which generalises the concept of a pure embedding of modules. Namely, an arrow $f\colon X\to Y$ is \emph{$\omega$-pure} if, given any commutative square as displayed below, with $A,B$ finitely presentable,
\[\begin{tikzcd}
A \arrow{d} \arrow{r} & X \arrow{d}{f} \\
B \arrow{r} \arrow[dashed]{ur} & Y
\end{tikzcd}\]
there is an arrow $B\to X$ making the upper triangle commute. 
By an \emph{$\omega$-pure subobject}, we mean a subobject represented by an $\omega$-pure morphism.

\begin{remark}\label{rm:purity}
  Let us call an object $X$ \emph{absolutely $\omega$-pure} if every monomorphism with domain $X$ is $\omega$-pure. In the category of $R$-modules, for $R$ any ring, the absolutely $\omega$-pure objects are the absolutely pure modules in the usual sense (see e.g.\ \cite{Maddox1967}). In the category of fields, they are the algebraically closed fields. See~\cite{Rothmaler1997} for a nice survey on this topic. As shown in \cite[Proposition~2.31]{ar94book}, every $\omega$-pure morphism in an lfp category is a regular monomorphism. Therefore, every absolutely $\omega$-pure object in an lfp category is absolutely closed and, a fortiori, saturated.
  \end{remark}

The following characterisation of $\omega$-injectivity classes in lfp categories was proved in \cite[Theorem~2.2]{RAB2022}.

\begin{theorem}\label{t:charact-inj-classes}
A full subcategory $\catt$ of an lfp category $\cat$ is an $\omega$-injectivity class if, and only if, it is closed in $\cat$ under products, directed colimits, and $\omega$-pure subobjects.
\end{theorem}

\section{Beth companions and saturated objects}\label{s:Beth-saturated}

Before stating the formal definition, we will motivate the concept of a Beth companion. If $\cat$ is any lfp category, its Beth companion $\B{\cat}$ should be a balanced full and replete subcategory of $\cat$ such that the inclusion functor $\B{\cat}\into \cat$ is a morphism of lfp categories.\footnote{The requirement that $\B{\cat}$ be a full and replete subcategory of $\cat$, rather than simply admitting a fully faithful functor to $\cat$, is not essential. However, it does simplify statements, since Beth companions are then unique rather than unique up to equivalence.} Moreover, to capture the ``companionship'' property, we require that every object of $\cat$ admits a monomorphism to an object of $\B{\cat}$. Since morphisms of lfp categories are right adjoints, this is equivalent to requiring that $\B{\cat}$ is a \emph{mono-reflective} subcategory of $\cat$, meaning that the components of the unit of the adjunction are monic (in fact, epi-monic; see e.g.\ \cite[Proposition~16.3]{ahs06book}). 

\begin{definition}\label{def:Beth-com}
Let $\cat$ be an lfp category. If it exists, a \emph{Beth companion of $\cat$} is a full and replete mono-reflective subcategory
\[\begin{tikzcd}[column sep=4em]
\B{\cat} \arrow[hookrightarrow]{r}[yshift=5pt]{\text{\tiny{$\bot$}}} & \cat \arrow[bend right=70,xshift=2pt]{l}
\end{tikzcd}\]
satisfying the following conditions:
\begin{enumerate}[label=(\alph*)]
\item $\B{\cat}$ is lfp and the inclusion functor $\B{\cat}\into \cat$ preserves directed colimits;
\item\label{i:ES-comp} $\B{\cat}$ satisfies the \ES~property (equivalently, it is balanced).
\end{enumerate}
\end{definition}
 
In other words, a Beth companion of $\cat$ is a (full and replete) balanced mono-reflective subcategory of $\cat$ in the category of lfp categories and lfp morphisms.

The main result of this section (see Theorem~\ref{t:Beth-companion-unique} below) shows that if $\cat$ is an lfp category that admits a Beth companion, then the latter coincides with $\Sat(\cat)$. The theorem also provides necessary and sufficient conditions for the existence of the Beth companion. We begin by working towards the first of these statements.

\begin{proposition}\label{p:saturated-Beth-comp-lfp}
Let $\catt$ be a balanced, full and replete mono-reflective subcategory of an lfp category $\cat$. Then $\catt$ coincides with $\Sat(\cat)$. 
\end{proposition}

To prove this fact, we start with the following observation.

\begin{lemma}\label{l:mono-refl-sat}
Let $\catt$ be a full, replete and mono-reflective subcategory of a category~$\cat$. Then $\catt$ contains $\Sat(\cat)$. If, in addition, $\catt$ is balanced and every object of $\cat$ admits an epi-mono to a saturated object, then $\catt$ is contained in $\Sat(\cat)$.
\end{lemma}
\begin{proof}
Let $R\colon \catt\into \cat$ be the inclusion functor, $L$ its left adjoint, and $\eta$ the unit of the adjunction.
For the first part of the statement, observe that the component $\eta_{A}$ of $\eta$ at an object~$A$ of $\cat$ is an epi-mono with domain $A$, hence $\eta_{A}$ is an isomorphism whenever $A$ is saturated. Since $\catt$ is a replete subcategory of $\cat$, it follows that every saturated object of $\cat$ belongs to $\catt$.

For the second part of the statement, it suffices to show that, for each object $A$ of~$\cat$, $RL(A)$ belongs to $\Sat(\cat)$. Consider an epi-mono $f\colon RL(A)\to E$ with $E$ in $\Sat(\cat)$. It follows from the previous part of the proof that $E$ belongs to $\catt$. Thus, since $\catt$ is balanced, $f$ is an isomorphism and so $RL(A)$ belongs to $\Sat(\cat)$.
\end{proof}

To deduce Proposition~\ref{p:saturated-Beth-comp-lfp} from Lemma~\ref{l:mono-refl-sat}, we need to show that every object of an lfp category admits an epi-mono to a saturated object.
To this end, given any object $A$ of a category $\cat$, we define a preorder relation $\preccurlyeq$ on the class of all \emph{extensions} of $A$, i.e.\ all monomorphisms in $\cat$ with domain $A$, as follows. If $m\colon A\emb B$ and $n\colon A \emb C$, we set
\[
m\preccurlyeq n \ \Longleftrightarrow \ \text{there is a monomorphism $h$ such that} \ n=h\circ m.
\]
The ensuing (possibly large) preorder is denoted by $\Ext(A)$, and the symmetrisation of $\preccurlyeq$ is denoted by $\sim$ (that is, $m\sim n$ if, and only if, $m\preccurlyeq n$ and $n\preccurlyeq m$).

\begin{definition}\label{def:epi-extension}
An \emph{epi-extension} of $A$ is any epi-mono in $\cat$ with domain $A$. Write $\EExt(A)$ for the sub-preorder of $\Ext(A)$ consisting of the epi-extensions of $A$. An element $m\in \EExt(A)$ is \emph{maximal} if, for any $n\in \EExt(A)$, $m\preccurlyeq n$ implies $m\sim n$.
\end{definition}

\begin{remark}\label{rem:order-on-epi-ext}
It is not difficult to see that, for any two epi-extensions $m,n\in \EExt(A)$,
\begin{enumerate}[label=(\roman*)]
\item\label{i:preceq-unique-epi-mono} $m\preccurlyeq n$ if, and only if, there is a (necessarily unique) epi-mono $h$ such that $n=h\circ m$;
\item\label{i:sym-equi} $m\sim n$ if, and only if, there is a (necessarily unique) isomorphism $h$ such that $n=h\circ m$;
\item\label{i:unique-epi-mono-is-iso} if $m\sim n$, then the unique epi-mono $h$ satisfying $n=h\circ m$ is an isomorphism.
\end{enumerate}
\end{remark}

\begin{lemma}\label{l:maximal-saturated}
Let ${m\colon A\emb E}$ be an epi-extension of $A$. The following statements are equivalent:
\begin{enumerate}[label=(\arabic*)]
\item\label{i:codomain-saturated} $E$ is saturated;
\item\label{i:maximal-in-EExt} $m$ is maximal in $\EExt(A)$.
\end{enumerate}
\end{lemma}

\begin{proof}
\ref{i:codomain-saturated} $\IMP$ \ref{i:maximal-in-EExt}. If $n\colon A\emb B$ is an epi-extension of $A$ satisfying $m\preccurlyeq n$, by Remark~\ref{rem:order-on-epi-ext}\ref{i:preceq-unique-epi-mono} there is an epi-mono $h\colon E\emb B$ such that $n=h\circ m$. As $E$ is saturated, $h$ is an isomorphism. Thus, $m\sim n$ by Remark~\ref{rem:order-on-epi-ext}\ref{i:sym-equi}.

\ref{i:maximal-in-EExt} $\IMP$ \ref{i:codomain-saturated}. Suppose $h\colon E\to B$ is an epi-mono. Since $m\preccurlyeq h\circ m$ in $\EExt(A)$, and $m$ is maximal, we get $m\sim h\circ m$. Thus, $h$ is an isomorphism by Remark~\ref{rem:order-on-epi-ext}\ref{i:unique-epi-mono-is-iso}.
\end{proof}

Together with Lemma~\ref{l:mono-refl-sat}, the next result, which is a consequence of Zorn's Lemma, completes the proof of Proposition~\ref{p:saturated-Beth-comp-lfp}.

\begin{lemma}\label{l:existence-saturated-epi-ext}
In an lfp category, every object admits a saturated epi-extension.
\end{lemma}

\begin{proof}
Let $A$ be an object of an lfp category, and let $P$ be the poset reflection of the preorder $\EExt(A)$ of epi-extensions of $A$. Further, let $\Epi(A)$ be the poset of subobjects of $A$ regarded as an object of~$\cat^{\op}$. By item~\ref{i:sym-equi} in Remark~\ref{rem:order-on-epi-ext}, there is an injective map $P\to \Epi(A)$. Since every lfp category is well-copowered \cite[Theorem~1.58]{ar94book}, $\Epi(A)$ is small, and thus so is $P$. Note that $P$ is non-empty because it contains the (equivalence class of) the identity of $A$. Moreover, every chain in $P$ admits an upper bound, which can be obtained by computing the colimit $A_{\omega}$ of the chain. The colimit map $A\to A_{\omega}$ is easily seen to be epic, and it is monic by \cite[Proposition~1.62]{ar94book}. By Zorn's Lemma, there exists a maximal element in $P$. In view of Lemma~\ref{l:maximal-saturated}, such a maximal element is (represented by) an epi-extension $A\to E$ with $E$ saturated. 
\end{proof}
 
The last missing ingredient to derive the main result of this section is the following elementary fact about mono-reflective subcategories.
\begin{lemma}\label{l:basic-props-mono-reflections}
Let $\catt$ be a full mono-reflective subcategory of a category $\cat$. The inclusion functor $R\colon\catt\into\cat$ preserves epis and reflects both extremal and regular monos,
and the reflector $L\colon \cat\to\catt$ is faithful.
\end{lemma}

\begin{proof}
Suppose that $e\colon X\to Y$ is an epimorphism in $\catt$, and consider arrows $f,g\colon R(Y)\rightrightarrows A$ in $\cat$ such that $f\circ R(e) = g\circ R(e)$. Composing with the component of the unit at $A$, i.e.\ $\eta_{A}\colon A\to RL(A)$, we get $\eta_{A}\circ f\circ R(e) = \eta_{A}\circ g\circ R(e)$. Since $\catt$ is a full subcategory of $\cat$, there exist arrows $u,v\colon Y\rightrightarrows L(A)$ in $\catt$ such that $R(u)=\eta_{A}\circ f$ and $R(v)=\eta_{A}\circ g$. It follows that 
\[
R(u\circ e) = R(u)\circ R(e)= R(v)\circ R(e) = R(v\circ e),
\]
and so $u\circ e = v\circ e$ in $\catt$ because $R$ is faithful. As $e$ is an epimorphism, we get $u=v$ and therefore $\eta_{A}\circ f = \eta_{A}\circ g$. Using the fact that $\eta_{A}$ is monic, we conclude that $f=g$, showing that $R(e)$ is epic.

It is straightforward to see that any fully faithful functor that preserves epis must reflect extremal monos. Hence, $R$ reflects extremal monos. 
To see that $R$ reflects regular monos, suppose that $f\colon X\to Y$ is an arrow in $\catt$ such that $R(f)$ is the equaliser of a pair of parallel arrows $u,v\colon R(Y)\rightrightarrows A$. 
Composing with the component $\eta_{A}$ of the unit at $A$, we obtain arrows $u'\coloneqq \eta_{A}\circ u$ and $v'\coloneqq \eta_{A}\circ v$. Since $\eta_{A}$ is monic, $R(f)$ is also the equaliser of $u'$ and $v'$. As $R$ is full, there are arrows $g,h\colon Y\rightrightarrows L(A)$ such that $R(g)=u'$ and $R(h)=v'$. Using the fact that fully faithful functors reflect limits, we conclude that $f$ is the equaliser of $g$ and $h$.

Finally, recall that, given an adjoint pair of functors, the left adjoint is faithful if, and only if, the unit of the adjunction is pointwise monic (see e.g.\ \cite[Theorem~19.14(1)]{ahs06book} for the dual statement). Thus, the reflector $L$ is faithful.
\end{proof}
 
\begin{theorem}\label{t:Beth-companion-unique} 
Let $\cat$ be an lfp category. If it exists, the Beth companion of $\cat$ is unique. Moreover, the following statements are equivalent:
\begin{enumerate}[label=(\arabic*)]
\item\label{i:has-Beth-comp} $\cat$ admits a Beth companion;
\item\label{i:Beth-comp-is-Sigma} $\Sat(\cat)$ is a Beth companion of $\cat$;
\item\label{i:Sigma-axiomatisable} $\Sat(\cat)$ is closed under limits and directed colimits in $\cat$.
\end{enumerate}
\end{theorem}

\begin{proof}
The uniqueness of the Beth companion and the implication \ref{i:has-Beth-comp} $\IMP$ \ref{i:Beth-comp-is-Sigma} are immediate consequences of Proposition~\ref{p:saturated-Beth-comp-lfp}.

\ref{i:Beth-comp-is-Sigma} $\IMP$ \ref{i:Sigma-axiomatisable}. Since $\Sat(\cat)$ is a reflective subcategory of an lfp category, it is complete and cocomplete. 
The inclusion functor $\Sat(\cat)\into\cat$ preserves limits and directed colimits, hence $\Sat(\cat)$ is closed under limits and directed colimits in $\cat$.

\ref{i:Sigma-axiomatisable} $\IMP$ \ref{i:has-Beth-comp}. By \cite[Corollary~2.4]{MP1987} (see also \cite[Corollary~2.48]{ar94book}), a full subcategory of an lfp category $\cat$ is itself lfp provided that it is closed in $\cat$ under limits and directed colimits. Hence, $\Sat(\cat)$ is lfp and the inclusion functor $\Sat(\cat)\into\cat$ is right adjoint by \cite[Theorem~1.66]{ar94book}. As every object of $\cat$ admits a saturated epi-extension by Lemma~\ref{l:existence-saturated-epi-ext}, $\Sat(\cat)$ is a (full and replete) mono-reflective subcategory of $\cat$. To conclude that $\Sat(\cat)$ is a Beth companion of $\cat$, it remains to show that $\Sat(\cat)$ has the \ES~property. In turn, this follows from the fact that the inclusion $\Sat(\cat)\into \cat$ reflects extremal monomorphisms by Lemma~\ref{l:basic-props-mono-reflections}.
\end{proof}

\begin{remark}\label{rmk:axiomatisability-Beth-comp}
Theorem~\ref{t:Beth-companion-unique} allows us to reduce the question of whether the Beth companion exists to an axiomatisability property of the class of saturated objects. This is akin to the concept of \emph{model companion} in model theory~\cite{Rob77}. Recall that a first-order theory has a model companion if its class of \emph{existentially closed} models is elementary (i.e., can be axiomatised in first-order logic). The conditions in item~\ref{i:Sigma-axiomatisable} of Theorem~\ref{t:Beth-companion-unique} can be regarded as axiomatisability properties of the class of saturated objects; in fact, by \cite[Corollary~2.4]{MP1987}, they imply that $\Sat(\cat)$ is the class of models of a finitary essentially algebraic theory. Because lfp categories do not come with a distinguished forgetful functor to the category of sets, we use closure under limits and directed colimits in place of  syntactic axiomatisability.

Continuing the analogy with model companions, it has been known since the 1970s that commutative rings do not have a model companion (see e.g.\ \cite[Corollary~57]{Cherlin1976}), but reduced and (von Neumann) regular commutative rings share the same model companion \cite{LS1973}. As we will see in Section~\ref{s:examples}, similar results hold for the Beth companions of these classes of rings.
\end{remark}

\section{Strong Beth companions and absolutely closed objects}\label{s:strong-Beth-sbs-closed}

In this section, we study the notion of \emph{strong Beth companion}, which is obtained by replacing the $\ES$ property in item~\ref{i:ES-comp} of Definition~\ref{def:Beth-com} with the \SES~property. 

\begin{remark}\label{rmk:comparison-Beth-companions}
Let $\K$ be a quasivariety with the amalgamation property that admits a Beth companion $\M$ in the sense of \cite{CKM2026bis}; in particular, this means that~$\M$ must also be a quasivariety. Then $\M$ is a strong Beth companion in our~sense. Indeed, Theorem~2.2 and Corollary~3.2 in \emph{op.~cit.} imply that $\M$ is (isomorphic to) a mono-reflective subcategory of $\K$ satisfying \SES. Considering the underlying-set functors $U_{\M}$ and $U_{\K}$, we obtain a commutative diagram as follows.
\[\begin{tikzcd}[column sep=1em, row sep=1.5em]
\M \arrow[hookrightarrow]{rr} \arrow{dr}[swap]{U_{\M}} & & \K \arrow{dl}{U_{\K}} \\
& \Set &
\end{tikzcd}\]
Since $\K$ and $\M$ are quasivarieties, $U_{\M}$ and $U_{\K}$ create directed colimits. Therefore, the inclusion $\M\into \K$ preserves directed colimits, and so $\M$ is a strong Beth companion of $\K$ in our sense.
\end{remark}

Observe that any lfp category $\cat$ admits at most one strong Beth companion. This follows from the fact that $\cat$ admits at most one Beth companion, and every strong Beth companion is a Beth companion. 
Moreover, $\Abs(\cat)$ is a Beth companion of~$\cat$ if, and only if, it is a strong Beth companion. Just observe that, if $\Abs(\cat)$ is a mono-reflective subcategory of~$\cat$, then the inclusion functor $\Abs(\cat)\into \cat$ reflects regular monos by Lemma~\ref{l:basic-props-mono-reflections}, and so $\Abs(\cat)$ satisfies \SES.
In fact, we can say more:

\begin{theorem}\label{t:strong-Beth-companion-unique}
Let $\cat$ be an lfp category. If it exists, the strong Beth companion of~$\cat$ is unique. Moreover, the following statements are equivalent:
\begin{enumerate}[label=(\arabic*)]
\item\label{i:has-strong-Beth-comp} $\cat$ admits a strong Beth companion;
\item\label{i:beth-comp-exists-and-sat-equal-abs} $\cat$ admits a Beth companion and $\Sat(\cat)=\Abs(\cat)$;
\item\label{i:abs-is-strong-beth-comp} $\Abs(\cat)$ is a (strong) Beth companion of $\cat$. 
\end{enumerate}
\end{theorem}

\begin{proof}
We already observed that, if they exist, strong Beth companions are unique.

\ref{i:has-strong-Beth-comp} $\IMP$ \ref{i:beth-comp-exists-and-sat-equal-abs}. It suffices to show that $\Sat(\cat)\subseteq\Abs(\cat)$, since the converse is always true. Note that, if $\cat$ admits a strong Beth companion, then the latter coincides with $\Sat(\cat)$ by Theorem~\ref{t:Beth-companion-unique}. Thus, $\Sat(\cat)$ is a strong Beth companion of~$\cat$.

Let $X$ be a saturated object of $\cat$, and let $m\colon X\emb Y$ be a mono in $\cat$. Since $\Sat(\cat)$ is a  mono-reflective subcategory of $\cat$, composing $m$ with the component $\eta_{Y}\colon Y \emb Y^{*}$ of the unit at $Y$, we obtain a monomorphism $g\coloneqq \eta_{Y}\circ m \colon X\emb Y^{*}$ in $\cat$ between saturated objects. As $\Sat(\cat)$ satisfies \SES, $g$ is a regular mono in $\Sat(\cat)$. The inclusion functor $\Sat(\cat)\into\cat$ preserves limits because it is right adjoint, so $g$ is a regular monomorphism in $\cat$. Since $\eta_{Y}$ is monic, it follows that $m$ is a regular mono. That is, $X$ is absolutely closed in $\cat$.

\ref{i:beth-comp-exists-and-sat-equal-abs} $\IMP$ \ref{i:abs-is-strong-beth-comp}. By Theorem~\ref{t:Beth-companion-unique}, if $\cat$ admits a Beth companion, then $\Sat(\cat)$ is a Beth companion of $\cat$. If $\Sat(\cat)=\Abs(\cat)$, then the latter is also a Beth companion of $\cat$. As observed above, $\Abs(\cat)$ is a Beth companion of $\cat$ precisely when it is a strong Beth companion. 

\ref{i:abs-is-strong-beth-comp} $\IMP$ \ref{i:has-strong-Beth-comp}. Clear.
\end{proof}

In particular, Theorem~\ref{t:strong-Beth-companion-unique} entails that an lfp category satisfying \ES~admits a strong Beth companion if, and only if, it satisfies \SES; equivalently, if and only if it is its own strong Beth companion. Therefore, if an lfp category satisfies \ES~but not \SES, it does not admit a strong Beth companion.

\section{Beth companions as orthogonality classes}\label{s:ort-classes-as-Beth-companions}

When $\cat$ is an lfp category satisfying a weak form of the amalgamation property, the problem of determining (the existence of) its Beth companion is tightly related to the problem of localising at the class of epi-monos, i.e.\ of inverting all epi-monos in $\cat$. This is made precise in Lemma~\ref{l:equivalent-problems} below. 

\begin{definition}\label{d:EAP}
A category has
\begin{enumerate}
\item[(\EAP)]\label{EAP} the \emph{epimorphism amalgamation property} if any span of arrows $\alpha,\beta$, with $\alpha$ and $\beta$ epi-monic, can be completed to a commutative square as displayed below, where $\gamma$ and $\delta$ are monic.
\[\begin{tikzcd}[cells={nodes={inner sep=2pt, outer sep=1pt}}]
\phantom{\cdot} \arrow{r}{\alpha} \arrow{d}[swap]{\beta} & \phantom{\cdot} \arrow[dashed]{d}{\delta} \\
\phantom{\cdot} \arrow[dashed]{r}{\gamma} & \phantom{\cdot}
\end{tikzcd}\]
\end{enumerate}
\end{definition}
Note that \AP~$\IMP$ \EAP. If a category has pushouts of epi-monos along epi-monos, then it has \EAP~if, and only if, epi-monos are stable under pushouts along epi-monos.

\begin{remark}\label{rmk:EAP-unique-saturated-epi-ext}
Recall from Definition~\ref{def:epi-extension} and Lemma~\ref{l:maximal-saturated} that $\EExt(A)$ denotes the preorder of epi-extensions of an object $A$, whose maximal elements are the saturated epi-extensions of $A$. Using the fact that every object of an lfp category admits a saturated epi-extension (Lemma~\ref{l:existence-saturated-epi-ext}), it is not difficult to see that an lfp category $\cat$ satisfies \EAP~if, and only if, for every object $A$ of $\cat$, the preorder $\EExt(A)$ has a top element. In particular, if $\cat$ satisfies \EAP, then each of its objects admits, up to isomorphism, exactly one saturated epi-extension.
\end{remark}

We start with a simple observation, which is an adaptation of~\cite[Lemma~1]{Ringel1971}.

\begin{lemma}\label{l:Ringel-restated}
Let $\cat$ be any category satisfying \EAP. For every full mono-reflective subcategory $\catt$ of $\cat$, the reflector $\cat \to \catt$ preserves epi-monos.
\end{lemma}
\begin{proof}
We start by proving that \EAP~implies the following 2-out-of-3 property for epi-monos: given a composite arrow $g\circ f$ in $\cat$, if both $g\circ f$ and $f$ are epi-monos, then so is $g$. Clearly, $g$ is always epic. To see that it is monic, note that \EAP~entails the existence of monomorphisms $u,v$ making the following diagram commute.
\[\begin{tikzcd}[cells={nodes={inner sep=2pt, outer sep=1pt}}]
\phantom{\cdot} \arrow{r}{f} \arrow{d}[swap]{g\circ f} & \phantom{\cdot} \arrow[dashed]{d}{u} \\
\phantom{\cdot} \arrow[dashed]{r}{v} & \phantom{\cdot}
\end{tikzcd}\]
As $f$ is epic, the equality $u\circ f = v\circ g\circ f$ implies $u= v\circ g$. Since $u$ is monic, so is $g$.

Now, let $R\colon \catt\into\cat$ be a full mono-reflective subcategory, with reflector $L\colon \cat\to\catt$, and write $\eta$ for the unit of the adjunction $L\dashv R$. For any arrow $f\colon X\to Y$ in $\cat$, there is a naturality square as displayed below.
\[\begin{tikzcd}[cells={nodes={inner sep=2pt, outer sep=1pt}}]
\phantom{\cdot} \arrow{r}{f} \arrow{d}[swap]{\eta_{X}} & \phantom{\cdot} \arrow{d}{\eta_{Y}} \\
\phantom{\cdot} \arrow{r}[swap]{RL(f)} & \phantom{\cdot}
\end{tikzcd}\]
Since $\catt$ is a mono-reflective subcategory of $\cat$, the components of $\eta$ are epi-monic. If~$f$ is an epi-mono, then so is the composite $\eta_{Y}\circ f = RL(f)\circ \eta_{X}$. It follows from the 2-out-of-3 property established above that $RL(f)$ is an epi-mono. The functor~$R$ reflects epi-monos because it is faithful; therefore $L(f)$ is epi-monic.
\end{proof}

\begin{lemma}\label{l:equivalent-problems}
Let $\cat$ be an lfp category satisfying \EAP.
The following statements are equivalent for any full mono-reflective subcategory $\catt \into \cat$:
\begin{enumerate}[label=(\arabic*)]
\item\label{i:catt-balanced} $\catt$ is balanced;
\item\label{i:inverts-epi-monos} the reflector $\cat\to\catt$ inverts all epi-monos.
\end{enumerate}
\end{lemma}

\begin{proof}
Let $R\colon \catt\into\cat$ be the inclusion functor, $L\colon \cat\to\catt$ the reflector, and $\eta$ the unit of the adjunction $L\dashv R$.

\ref{i:catt-balanced} $\IMP$ \ref{i:inverts-epi-monos}. If $f$ is an epi-mono in $\cat$, then $L(f)$ is an epi-mono in $\catt$ by Lemma~\ref{l:Ringel-restated}. As $\catt$ is balanced, it follows that $L(f)$ is an isomorphism.

\ref{i:inverts-epi-monos} $\IMP$ \ref{i:catt-balanced}. It suffices to show that the inclusion $R\colon \catt\into\cat$ preserves epi-monos. In turn, this follows at once from the fact that $R$ preserves epimorphisms by Lemma~\ref{l:basic-props-mono-reflections}, and it preserves monomorphisms because it is right adjoint.
\end{proof}

Given a category $\cat$, write $\X$ for the class of all epi-monos in~$\cat$. Let $\ort{\X}$ be the associated orthogonality class, i.e., the full subcategory of $\cat$ consisting of those objects that are orthogonal to all epi-monos in $\cat$. 

\begin{lemma}\label{l:ort-included-in-sigma}
For any category $\cat$, we have $\ort{\X}\subseteq \Sat(\cat)$.
\end{lemma}

\begin{proof}
If $X\in \ort{\X}$, and $m\colon X\to Y$ is an epi-mono in $\cat$, then there is an arrow $r$ making the following diagram commute.
\[\begin{tikzcd}
X \arrow{d}[swap]{m} \arrow{r}{\id_{X}} & X \\
Y \arrow[dashed]{ur}[swap]{r}
\end{tikzcd}\]
It follows that $m$ is both a section and an epimorphism, hence an isomorphism. Thus, $X$ is saturated.
\end{proof}

We now assume that $\cat$ is an lfp category. Under this assumption, $\ort{\X}$ is a reflective subcategory of $\cat$ by Proposition~\ref{p:localise-epis}. 

\begin{proposition}\label{p:localisation-is-X-ort}
For any full and replete mono-reflective subcategory $\catt\into \cat$ of an lfp category~$\cat$, we have $\ort{\X}\subseteq \catt$. If, in addition, $\cat$ satisfies \EAP~and $\catt$ is balanced, then $\ort{\X}= \catt$.
\end{proposition}
\begin{proof}
Let $\catt\into \cat$ be a mono-reflective subcategory, with reflector $L$. Every full and replete reflective subcategory is the localisation at the collection of arrows inverted by the reflector; thus we can identify $\catt$ with the localisation of $\cat$ at $\ar{W}\coloneqq L^{-1}(\text{isos})$. 

The functor $L$ is faithful by Lemma~\ref{l:basic-props-mono-reflections}, so it reflects epis and monos. It follows that $f$ is epi-mono whenever $L(f)$ is an isomorphism. That is, $\ar{W}\subseteq \X$.
As the reflector $\cat \to \ort{\X}$ inverts all epi-monos, hence all arrows in $\ar{W}$, we get $\ort{\X}\subseteq \catt$.

If, in addition, $\cat$ satisfies \EAP~and $\catt$ is balanced, then $\X\subseteq \ar{W}$ by Lemma~\ref{l:equivalent-problems}. It follows that $\ort{\X}= \catt$.
\end{proof}

\begin{theorem}\label{t:Beth-companion-amalgam-orthogonal}
Let $\cat$ be an lfp category satisfying \EAP.
The following statements are equivalent:
\begin{enumerate}[label=(\arabic*)]
\item\label{i:cat-has-Beth-comp-amal} $\cat$ admits a Beth companion;
\item\label{i:ort-X-is-Beth-comp} $\ort{\X}$ is a Beth companion of $\cat$;
\item\label{i:ort-X-closed-dir-colim} $\ort{\X}$ is closed under directed colimits in $\cat$, and every object of $\cat$ admits an extension with codomain in $\ort{\X}$.
\end{enumerate}
\end{theorem}

\begin{proof}
The implication \ref{i:cat-has-Beth-comp-amal} $\IMP$ \ref{i:ort-X-is-Beth-comp} is an immediate consequence of Proposition~\ref{p:localisation-is-X-ort}.

The implications \ref{i:ort-X-is-Beth-comp} $\IMP$ \ref{i:ort-X-closed-dir-colim} and \ref{i:ort-X-closed-dir-colim}  $\IMP$ \ref{i:cat-has-Beth-comp-amal} are completely analogous to the corresponding ones in Theorem~\ref{t:Beth-companion-unique}. In particular, for \ref{i:ort-X-closed-dir-colim}  $\IMP$ \ref{i:cat-has-Beth-comp-amal}, observe that $\ort{\X}$ is always closed under limits in $\cat$ (see e.g.\ \cite[Observation~1.34]{ar94book}), and the assumption that every object of $\cat$ admits an extension with codomain in $\ort{\X}$ implies that $\ort{\X}$ is a \emph{mono}-reflective subcategory of~$\cat$. To show that $\ort{\X}$ is a balanced category, we reason as in the proof of Theorem~\ref{t:Beth-companion-unique}, using the fact that $\ort{\X}\subseteq \Sat(\cat)$ by Lemma~\ref{l:ort-included-in-sigma}.
\end{proof}

In view of Theorem~\ref{t:Beth-companion-amalgam-orthogonal}, it is useful to have criteria ensuring that $\ort{\X}$ is closed under directed colimits in $\cat$ and that every object of $\cat$ can be extended to an object of~$\ort{\X}$, as these properties imply that $\ort{\X}$ is the Beth companion of $\cat$ (assuming the latter category satisfies \EAP).
To characterise when every object of $\cat$ can be extended to an object of $\ort{\X}$, we consider the following notion.

\begin{definition}\label{d:TEM}
A category has 
\begin{enumerate}
\item[(TEM)]\label{TEM} \emph{transferable epi-monos} if any span of arrows $\alpha,\beta$, with $\alpha$ epi-monic, can be completed to a commutative square as displayed below, where $\gamma$ is epi-monic.
\[\begin{tikzcd}[cells={nodes={inner sep=2pt, outer sep=1pt}}]
\phantom{\cdot} \arrow{r}{\alpha} \arrow{d}[swap]{\beta} & \phantom{\cdot} \arrow[dashed]{d}{\delta} \\
\phantom{\cdot} \arrow[dashed]{r}[rightarrowtail]{\gamma} & \phantom{\cdot}
\end{tikzcd}\]
\end{enumerate}
\end{definition}
If a category has pushouts of epi-monos along all arrows, then it has transferable epi-monos precisely when epi-monos are pushout stable. In that case, \TEM~$\IMP$~\EAP.

\begin{lemma}\label{l:ort-vs-sat}
The following statements are equivalent for any lfp category $\cat$:
\begin{enumerate}[label=(\arabic*)]
\item\label{i:X-local-extension} every object of $\cat$ admits an extension with codomain in $\ort{\X}$;
\item\label{i:sat-and-ort-coincide} $\Sat(\cat)=\ort{\X}$;
\item\label{i:transferable-epi-monos} $\cat$ has transferable epi-monos.
\end{enumerate}
\end{lemma}

\begin{proof}
\ref{i:X-local-extension} $\IFF$ \ref{i:sat-and-ort-coincide}. As every object of an lfp category has a saturated (epi-)extension by Lemma~\ref{l:existence-saturated-epi-ext}, the implication \ref{i:sat-and-ort-coincide} $\IMP$ \ref{i:X-local-extension} is immediate. For the converse implication, suppose that every object of $\cat$ admits an extension with codomain in $\ort{\X}$. Then $\ort{\X}$ is a mono-reflective subcategory of $\cat$. The unit of this mono-reflection is component-wise epi-monic, hence its component at any saturated object is an isomorphism. It follows that $\Sat(\cat)\subseteq \ort{\X}$, and so $\Sat(\cat)=\ort{\X}$ by Lemma~\ref{l:ort-included-in-sigma}.

\ref{i:sat-and-ort-coincide} $\IFF$ \ref{i:transferable-epi-monos}.
Suppose that $\ort{\X} = \Sat(\cat)$. We claim that $\cat$ has transferable epi-monos. Consider an epi-mono $m\colon Y\to W$ and an arrow $f\colon Y\to X$. Since $\cat$ is lfp, $X$ admits a saturated epi-extension $n\colon X\to Z$. As $\Sat(\cat)\subseteq \ort{\X}$, there is an arrow $g$ making the following diagram commute. Thus, $\cat$ has transferable epi-monos.
\[\begin{tikzcd}
Y \arrow{d}[swap]{m} \arrow{r}{f} & X \arrow{d}{n} \\
W \arrow[dashed]{r}{g} & Z
\end{tikzcd}\]

Conversely, we claim that if $\cat$ has transferable epi-monos, then $\ort{\X} = \Sat(\cat)$ (for this, we do not need to assume that $\cat$ is lfp). 
Suppose that $X$ is a saturated object of $\cat$, and consider a span $f,m$ as displayed below, with $m$ epi-mono.
\[\begin{tikzcd}
Y \arrow{d}[swap]{m} \arrow{r}{f} & X \arrow[dashed]{d}{n} \\
W \arrow[dashed]{r}{h} & Z
\end{tikzcd}\]
Since $\cat$ has transferable epi-monos, there exist arrows $h,n$, with $n$ epi-mono, making the above square commute. As $X$ is saturated, $n$ is an isomorphism. Hence, $g\coloneqq n^{-1}\circ h\colon W\to X$ satisfies $g\circ m = f$, showing that $X$ is orthogonal to~$m$. This proves that $\Sat(\cat)\subseteq \ort{\X}$; the other inclusion follows from Lemma~\ref{l:ort-included-in-sigma}.
\end{proof}

\begin{remark}\label{rmk:EAP-vs-TEM}
Let $\cat$ be an lfp category satisfying \EAP. It follows from Theorem~\ref{t:Beth-companion-amalgam-orthogonal} and Lemma~\ref{l:ort-vs-sat} that, if $\cat$ has a Beth companion, then it satisfies the strengthening \TEM~of \EAP. In other words, \TEM~and \EAP~are equivalent for any lfp category that admits a Beth companion.
\end{remark}

Next, we characterise when $\ort{\X}$ is closed under directed colimits in $\cat$, assuming that $\cat$ has transferable epi-monos.

\begin{lemma}\label{l:X-ort-dir-colim-fp-domains}
Let $\cat$ be an lfp category satisfying \TEM. The following are equivalent:
\begin{enumerate}[label=(\arabic*)]
\item\label{i:ort-X-closed-dir-col} $\ort{\X}$ is closed under directed colimits in $\cat$;
\item\label{i:saturation-epimonos-fp-domains} $\ort{\X}=\X^{\perp}_{\mathrm{fpd}}$, where $\X_{\mathrm{fpd}}$ is the class of epi-monos in $\cat$ with finitely presentable domains.
\end{enumerate}
\end{lemma}

\begin{proof}
\ref{i:ort-X-closed-dir-col} $\IMP$ \ref{i:saturation-epimonos-fp-domains}. Let $A$ be any object of $\cat$, and let $D_{A}$ be the canonical (directed) diagram of finitely presentable objects of $\cat$ satisfying $A\cong \colim{D_{A}}$. Recall that $\ort{\X}$ is mono-reflective in $\cat$ by Proposition~\ref{p:localise-epis} and Lemma~\ref{l:ort-vs-sat}. Write $R\colon \ort{\X}\into \cat$ for the inclusion functor, $L$ for its left adjoint, and $\eta$ for the unit of this adjunction. Then $\eta_{A}$ is the colimit, in the category of arrows of $\cat$, of the arrows~$\eta_{B}$, where $B$ ranges over the image of the diagram $D_{A}$. Just observe that colimits in the arrow category are computed componentwise, and both $L$ and $R$ preserve directed colimits. 

The class of arrows inverted by the reflector $\cat \to \X^{\perp}_{\mathrm{fpd}}$ contains each arrow of the form $\eta_{B}$ and is closed under colimits in the arrow category; therefore, it contains all components of the unit $\eta$. It follows that $\ort{\X}$, which can be identified with the localisation at the components of $\eta$, coincides with $\X^{\perp}_{\mathrm{fpd}}$.

\ref{i:saturation-epimonos-fp-domains} $\IMP$ \ref{i:ort-X-closed-dir-col}. Each arrow in $\X_{\mathrm{fpd}}$ is an epimorphism with a finitely presentable domain, hence $\X^{\perp}_{\mathrm{fpd}}$ is closed in $\cat$ under directed colimits (this follows from a straightforward adaptation of \cite[Proposition~1.35]{ar94book}). Thus, so is $\ort{\X}$.
\end{proof}

The following result is a consequence of Theorem~\ref{t:Beth-companion-amalgam-orthogonal} and Lemma~\ref{l:ort-vs-sat}.

\begin{corollary}\label{cor:Beth-TEM-omega-ort}
Let $\cat$ be an lfp category with transferable epi-monos. The following statements are equivalent:
\begin{enumerate}[label=(\arabic*)]
\item\label{i:TEM-has-BC} $\cat$ admits a Beth companion;
\item\label{i:TEM-Sat-dir-col} $\Sat(\cat)$ ($=\ort{\X}$) is closed under directed colimits in $\cat$;
\item\label{i:TEM-Sat-inj-class} $\Sat(\cat)$ ($=\ort{\X}$) is an $\omega$-injectivity class in $\cat$. 
\end{enumerate}
\end{corollary}

\begin{proof}
By Lemma~\ref{l:ort-vs-sat}, we can assume throughout the proof that $\Sat(\cat)=\ort{\X}$.

\ref{i:TEM-has-BC} $\IMP$ \ref{i:TEM-Sat-dir-col}. This is an immediate consequence of Theorem~\ref{t:Beth-companion-unique}. 

\ref{i:TEM-Sat-dir-col} $\IMP$ \ref{i:TEM-Sat-inj-class}. By Theorem~\ref{t:charact-inj-classes}, it suffices to prove that $\Sat(\cat)$ is closed in $\cat$ under $\omega$-pure subobjects; just recall that $\ort{\X}$ is always closed in $\cat$ under limits, and thus so is $\Sat(\cat)$. Every $\omega$-pure morphism in an lfp category is a regular monomorphism (see e.g.\ \cite[Proposition~2.31]{ar94book}), hence an extremal monomorphism. Therefore, it is enough to show that the class of saturated objects in $\cat$ is closed under extremal subobjects. Consider an extremal mono $m\colon X \to Y$ in $\cat$, with $Y$ saturated, and let $f\colon X\to W$ be an epi-mono. We must prove that $f$ is an isomorphism. By \TEM, the span formed by $m$ and $f$ can be completed into a commutative square as displayed below, with $g$ epi-monic.
\[\begin{tikzcd}
X \arrow{r}{m} \arrow{d}[swap]{f} & Y \arrow{d}{g} \\
W \arrow{r}{h} & Z
\end{tikzcd}\]
Since $Y$ is saturated, $g$ is an isomorphism, and therefore $h\circ f = g\circ m$ is an extremal mono. It follows that $f$ is an extremal mono, hence an isomorphism.

\ref{i:TEM-Sat-inj-class} $\IMP$ \ref{i:TEM-has-BC}. Every $\omega$-injectivity class in $\cat$ is closed under directed colimits in $\cat$ by Theorem~\ref{t:charact-inj-classes}, so $\Sat(\cat)=\ort{\X}$ is a Beth companion of $\cat$ by Theorem~\ref{t:Beth-companion-amalgam-orthogonal} and Lemma~\ref{l:ort-vs-sat}.
\end{proof}

The following simple consequence of the previous results is particularly useful for constructing Beth companions. See, for example, the cases of bounded distributive lattices and regular rings in Section~\ref{s:examples}.

\begin{corollary}\label{cor:S-perp-Beth-comp}
Let $\cat$ be an lfp category with transferable epi-monos, and let $\ar{S}$ be a class of epi-monos in $\cat$ with finitely presentable domains. If $\ort{\ar{S}}$ is balanced, then it is the Beth companion of~$\cat$.
\end{corollary}

\begin{proof}
If $\ar{S}$ is any class of epi-monos between objects of $\cat$, then $\ort{\X}\subseteq \ort{\ar{S}}$. By Proposition~\ref{p:localise-epis}, $\ort{\ar{S}}$ is a reflective subcategory of $\cat$; in fact, it is mono-reflective because every object of $\cat$ admits an extension with codomain in $\ort{\X}$ (hence, in~$\ort{\ar{S}}$) by Lemma~\ref{l:ort-vs-sat}.
If $\ort{\ar{S}}$ is balanced, then $\ort{\ar{S}}=\ort{\X}$ by Proposition~\ref{p:localisation-is-X-ort}. If, in addition, the domain of each arrow in $\ar{S}$ is finitely presentable, then $\ort{\ar{S}}$ is closed under directed colimits in $\cat$ (this follows from a straightforward adaptation of \cite[Proposition~1.35]{ar94book}, using the fact that each arrow in $\ar{S}$ is an epimorphism). 
Therefore, $\ort{\ar{S}}$ is the Beth companion of $\cat$ by Theorem~\ref{t:Beth-companion-amalgam-orthogonal} and Corollary~\ref{cor:Beth-TEM-omega-ort}.
\end{proof}

We conclude this section by discussing the relationship between Beth companions and factorisation systems in the presence of \TEM. This remark is due to Michel H\'ebert (private communication), who kindly gave us permission to include it.

\begin{remark}\label{Hebert}
  Let $\cat$ be an lfp category satisfying \TEM, and $\X$ the class of epi-monos in $\cat$. Denote by $\ar{M}$ the class of arrows in $\cat$ having the \emph{right lifting property} with respect to all arrows in $\X$ (see e.g. \cite[\S 11.1]{Riehl2014}). The following statements hold:
  \begin{enumerate}[label=(\roman*)]
    \item\label{i:ortho-X-replete-mono-reflective} $\ort{\X}$ is a full replete mono-reflective subcategory of $\cat$ if, and only if, $(\X,\ar{M})$ is a factorisation system in $\cat$.
    \item\label{i:Beth-comp-factorisation-system} If $(\X,\ar{M})$ is a factorisation system in $\cat$ and $\ar{M}$ is closed under directed colimits in the category of arrows of $\cat$, then $\cat$ admits a Beth companion.
  \end{enumerate}
 The left-to-right implication in item~\ref{i:ortho-X-replete-mono-reflective} follows from \cite[Corollary~3.4]{CHK1985}, combined with the fact that $\X$ coincides with the class of arrows inverted by the reflector $\cat \to \ort{\X}$. The right-to-left implication follows from the considerations in~\cite[\S 2]{CHK1985}. Item~\ref{i:Beth-comp-factorisation-system} is a consequence of Corollary~\ref{cor:Beth-TEM-omega-ort}, since $\ort{\X}$ is closed under directed colimits in $\cat$ whenever $\ar{M}$ is closed under directed colimits in the arrow category.
\end{remark}

\section{Transfer of properties}\label{s:transfer-properties}

Throughout this section, $\cat$ denotes an lfp category that admits a Beth companion~$\B{\cat}$. We shall look at how the properties of $\cat$ carry over to~$\B{\cat}$. We start by recalling a categorical characterisation of finitary quasivarieties, which hinges on the existence of enough finitely presentable regular projective objects (such as free finitely generated algebras in a quasivariety). Recall that an object $X$ of a locally small category $\catt$ is \emph{regular projective} if the functor $\catt(X,-)\colon \catt\to\Set$ preserves regular epimorphisms. The following is \cite[Theorem~3.24]{ar94book}.

\begin{theorem}\label{t:characterisation-quasi-varieties}
A category is equivalent to a finitary (possibly multi-sorted) quasivariety if, and only if, it is cocomplete and admits a dense\footnote{For our purposes, a dense subcategory of a category $\catt$ can be defined as a full subcategory $\sf E$ of $\catt$ such that the restricted Yoneda embedding $\catt \to [\sf{E}^{\op},\Set]$ is fully faithful.} subcategory consisting of finitely presentable regular projectives. 
\end{theorem}

\begin{remark}
To characterise mono-sorted quasivarieties, it suffices to require that the dense subcategory in the above statement consists of a single finitely presentable regular projective object. 
\end{remark}

\begin{theorem}\label{t:transfer-quasi-variety}
Suppose $\cat$ satisfies \EAP, and there is a class $\ar{S}$ of epi-monos in $\cat$ with finitely presentable regular projective domains such that $\ort{\ar{S}}$ is balanced. If~$\cat$ is equivalent to a (mono-sorted) finitary (quasi)variety, then so is its Beth companion.
\end{theorem}

\begin{proof}
Let $\ar{S}$ be as in the statement. Upon recalling that $\cat$ satisfies \TEM~(see Remark~\ref{rmk:EAP-vs-TEM}), it follows from Corollary~\ref{cor:S-perp-Beth-comp} that $\ort{\ar{S}}$ is the Beth companion of~$\cat$. Let $R\colon \ort{\ar{S}}\into \cat$ denote the inclusion functor, and $L$ its left adjoint. 

By Theorem~\ref{t:characterisation-quasi-varieties}, it suffices to show that if $\catt$ is a dense subcategory of $\cat$ consisting of finitely presentable regular projectives, then the full subcategory $\catt_{L}$ of~$\ort{\ar{S}}$ formed by the objects of the form $L(X)$, for $X$ in~$\catt$, is a dense subcategory consisting of finitely presentable regular projectives.

To show that $\catt_{L}$ is dense in $\ort{\ar{S}}$, we must prove that the restricted Yoneda embedding $\ort{\ar{S}} \to [\catt_{L}^{\op},\Set]$ is fully faithful. That is, for every $Y_{1},Y_{2}$ in $\ort{\ar{S}}$ the canonical map
\[
\ort{\ar{S}}(Y_{1},Y_{2}) \to \mathrm{Nat}(\ort{\ar{S}}(-,Y_{1})_{\mid \catt_{L}}, \ort{\ar{S}}(-,Y_{2})_{\mid \catt_{L}})
\]
is a bijection. Let $\tau$ be a natural transformation in the rightmost set above. The component of $\tau$ at any object $L(X)$ of $\catt_{L}$ is 
\[
\tau_{L(X)}\colon \ort{\ar{S}}(L(X),Y_{1})\to \ort{\ar{S}}(L(X),Y_{2}),
\]
which, in view of the adjunction $L\dashv R$, yields a map
\[
\cat(X, R(Y_{1}))\xrightarrow{\cong} \ort{\ar{S}}(L(X),Y_{1}) \xrightarrow{\tau_{L(X)}} \ort{\ar{S}}(L(X),Y_{2}) \xrightarrow{\cong} \cat(X, R(Y_{2})).
\]
Let us denote the latter composite by $\sigma_{X}$. The naturality of $\tau$ entails that the maps~$\sigma_{X}$ are the components of a natural transformation 
\[
\sigma\colon \cat(-,R(Y_{1}))_{\mid \catt}\to \cat(-,R(Y_{2}))_{\mid \catt}.
\]
The restricted Yoneda embedding $\cat \to [\catt^{\op},\Set]$ is fully faithful because $\catt$ is dense in $\cat$. Hence, the natural transformation $\sigma$ is induced by a unique arrow $f\colon R(Y_{1})\to R(Y_{2})$. Since $R$ is the inclusion of a full subcategory, the arrow $f\colon Y_{1}\to Y_{2}$ lies in $\ort{\ar{S}}$. Tracing $f$ through the adjunction, we see that it is the unique arrow inducing the natural transformation $\tau$. Therefore, $\catt_{L}$ is dense in $\ort{\ar{S}}$.

Now, the reflector $L$ preserves finitely presentable objects because $R$ preserves directed colimits; hence, each object in $\catt_{L}$ is finitely presentable. To show that each object in $\catt_{L}$ is regular projective, we start by proving the following fact.

\begin{claim}
The inclusion $R\colon \ort{\ar{S}} \into \cat$ preserves regular epimorphisms.
\end{claim}

\begin{proof}[Proof of Claim]
We first show that $\ort{\ar{S}}$ is closed in $\cat$ under regular images. Let $q\colon A\to B$ be a regular epimorphism in $\cat$ with $A\in \ort{\ar{S}}$. We must show that $B\in \ort{\ar{S}}$, i.e.\ $B$ is orthogonal to each arrow $f\colon X \to Y$ in $\ar{S}$. If $g\colon X\to B$ is an arrow in $\cat$, since $X$ is  regular projective there is $h\colon X\to A$ such that $g = q\circ h$.
\[\begin{tikzcd}
{} & A \arrow{d}{q} \\
X \arrow{r}{g} \arrow{d}[swap]{f} \arrow[dashed]{ur}{h} & B \\
Y \arrow[dashed]{ur}[swap]{\ell} & {}
\end{tikzcd}\]
Because $A\in \ort{\ar{S}}$, there is $k\colon Y\to A$ such that $h = k\circ f$. Then $\ell\coloneqq q\circ k$ satisfies
\[
\ell \circ f = q\circ k \circ f = q\circ h = g,
\]
showing that $B\in \ort{\ar{S}}$.

Next, we show that $R$ sends epimorphisms to regular epimorphisms, thus settling the claim. Let $e\colon X\to Y$ be an epimorphism in $\ort{\ar{S}}$. Because every object of $\cat$ admits a monomorphism to an object of $\ort{\ar{S}}$, $R(e)$ is an epimorphism in $\cat$. Since $\cat$ is equivalent to a quasivariety, it admits a (regular epi, mono) factorisation system. Consider a decomposition of $R(e)$ as a regular epimorphism $q\colon R(X)\to A$ followed by a monomorphism $m\colon A\to R(Y)$. The object $A$ belongs to $\ort{\ar{S}}$ because the latter is closed in $\cat$ under regular images, and $m$ is an epi-mono because $R(e)$ is epi. Each object of $\ort{\ar{S}}$ is saturated (e.g.\ by Proposition~\ref{p:saturated-Beth-comp-lfp}), therefore $m$ is an isomorphism. We conclude that $R(e)$ is a regular epimorphism.
\end{proof}

If $X$ is any object of $\cat$, we have a natural isomorphism $\ort{\ar{S}}(L(X),-)\cong \cat(X,R(-))$. Since $R$ preserves regular epimorphisms by the previous claim, if $X$ is regular projective (i.e., if $\cat(X,-)\colon \cat\to\Set$ preserves regular epimorphisms) then so is $LX$. Thus, $\catt_{L}$ consists of regular projective objects.

This shows that $\ort{\ar{S}}$ is equivalent to a finitary quasivariety. Note that if $\cat$ is equivalent to a mono-sorted finitary quasivariety, then we can choose the dense subcategory $\catt$ to consist of a single object. The same proof as above then shows that $\ort{\ar{S}}$ is equivalent to a mono-sorted finitary quasivariety.

It remains to show that if $\cat$ is a variety, then so is $\ort{\ar{S}}$. Recall that a quasivariety is a variety if, and only if, it has effective equivalence relations (cf.\ e.g.\ \cite[Corollary~3.25]{ar94book}). Therefore, we must prove that if equivalence relations in $\cat$ are effective, the same holds in~$\ort{\ar{S}}$. Let $(r_{0},r_{1})\colon E\rightrightarrows X$ be an equivalence relation in~$\ort{\ar{S}}$. Because $R\colon \ort{\ar{S}}\into \cat$ preserves limits, $(R(r_{0}),R(r_{1}))$ is an equivalence relation in~$\cat$. As equivalence relations in the latter category are effective, $(R(r_{0}),R(r_{1}))$ is the kernel pair of its coequaliser $q\colon R(X)\to A$. The coequaliser of $(r_{0},r_{1})$ in $\ort{\ar{S}}$ is the unique arrow $\tilde{q}\colon X\to L(A)$ such that $R(\tilde{q}) = \eta_{A}\circ q$, where $\eta_{A}\colon A \to RL(A)$ is the component at $A$ of the unit of the adjunction $L\dashv R$. Because $\eta_{A}$ is monic, the kernel pair of $\eta_{A}\circ q$ in $\cat$ coincides with the kernel pair of $q$, hence with $(R(r_{0}),R(r_{1}))$. Since limits in $\ort{\ar{S}}$ are computed in $\cat$, it follows that the kernel pair of $\tilde{q}$ is $(r_{0},r_{1})$. We conclude that the latter equivalence relation is effective.
\end{proof}

Next, we look at how the amalgamation property transfers from $\cat$ to $\B{\cat}$.

\begin{proposition}\label{p:transfer-AP}
Suppose that $\cat$ satisfies the amalgamation property. Then $\B{\cat}$ has the strong amalgamation property.
\end{proposition}

\begin{proof}
Since $\B{\cat}$ is a reflective subcategory of $\cat$, pushouts in the former category are computed by first taking the pushout in $\cat$ and then applying the reflector. 
In view of \cite[Lemma~1]{Ringel1971}, given any full and replete mono-reflective subcategory of an lfp category that satisfies the amalgamation property, the reflector preserves monomorphisms. Hence, the reflector $\cat\to\B{\cat}$ preserves monos. The inclusion functor $\B{\cat}\into\cat$ preserves monos too, because it is right adjoint. Thus, since $\cat$ satisfies \AP, so does $\B{\cat}$. As \ES~and \SES~are equivalent in the presence of \AP, it follows that $\B{\cat}$ satisfies \SES, and thus also \SAP~(cf.\ Proposition~\ref{l:IPA-iff-SES}).
\end{proof}

We conclude this section with two remarks concerning the properties of $\B{\cat}$ when $\cat$ satisfies the amalgamation property or variations thereof.

\begin{remark}\label{r:TM}
Recall that a category $\catt$ has
\begin{enumerate}
\item[(TM)]\label{TM} \emph{transferable monomorphisms} if any span $\alpha,\beta$, with $\alpha$ monic, can be completed to a commutative square as in eq.~\eqref{eq:generic-amalgam} with $\gamma$ monic.
\end{enumerate}
If $\catt$ admits pushouts of monomorphisms along any morphism, then it has TM if, and only if, monomorphisms are stable under pushouts along any morphism. As in the proof of Proposition~\ref{p:transfer-AP}, an application of \cite[Lemma~1]{Ringel1971} shows that if $\cat$ satisfies \TM~(and, a fortiori, \AP), then so does its Beth companion.
\end{remark}

\begin{remark}
Let $\DL$ be the category of bounded distributive lattices and bounded lattice homomorphisms, and $\BA$ the full subcategory of $\DL$ consisting of Boolean algebras. The inclusion $\BA\into \DL$ has a left adjoint, which assigns to a distributive lattice its \emph{Boolean envelope}. A remarkable property of this adjunction is that, if we embed a distributive lattice $A$ into any Boolean algebra $B$, the Boolean subalgebra $\langle A\rangle$ of $B$ generated by (the image of) $A$ is isomorphic to the Booleanisation of $A$. Note that the inclusion $A\into \langle A\rangle$ is an epi-mono in $\DL$. 

Let us define the latter property for any reflective subcategory $R\colon \sf{E}\into \catt$, with left adjoint $L$ and unit $\eta$. We shall say that such an adjunction has the \emph{free extension property} if, given any epi-mono $f\colon A\to R(X)$ in $\catt$, the unique arrow $g\colon L(A)\to X$ such that the following diagram commutes is an isomorphism.
\[\begin{tikzcd}
A \arrow{r}{\eta_{A}} \arrow{dr}[swap]{f} & RL(A) \arrow{d}{R(g)} \\
{} & R(X)
\end{tikzcd}\]

It is a simple observation that, if $\cat$ has the amalgamation property, the mono-reflective subcategory $\B{\cat}\into\cat$ has the free extension property. Just observe that, in any category satisfying \AP, epi-monos are \emph{essential monos} (recall that a mono~$m$ is an essential mono if any composite $h\circ m$ is monic just when $h$ is monic); for a proof, see \cite[\S 2]{Ringel1971}. It follows that the components $\eta_{A}$ of the unit are essential monos, and so $R(g)$ is epi-monic whenever $f$ is epi-monic. As $R$ reflects epis and monos because it is faithful, $g$ is an epi-mono in $\B{\cat}$, hence an isomorphism.

We recover the free extension property for Boolean envelopes because $\DL$ satisfies the amalgamation property, and $\BA$ is the Beth companion of $\DL$ (cf.\ Section~\ref{s:examples}).
\end{remark}

\section{A syntactic perspective}\label{s:syntactic-viewpoint}

Under Gabriel--Ulmer duality for lfp categories, to any lfp category $\cat$, one can associate an essentially unique theory whose class of models is equivalent to $\cat$ (more details will be provided below). If $\cat$ admits a Beth companion $\B{\cat}$, it is natural to wonder what the relationship is between the theory of $\cat$ and that of $\B{\cat}$.

Let $\X$ be the class of epi-monos in $\cat$, and write $\mathbb{T}_{\X}$ for the idempotent monad on~$\cat$ induced by the reflective subcategory $\ort{\X}\into \cat$. Recall that a monad is \emph{strongly finitary} if it preserves finitely presentable objects.\footnote{The terminology \emph{strongly finitary} is sometimes used in the literature to refer to a different property; see e.g.\ \cite{KL1993,ADV2025}. Our usage is consistent with that of, e.g., \cite{AS2003}.} The main result of this section shows that, if $\cat$ satisfies \EAP~and $\mathbb{T}_{\X}$ is strongly finitary, then $\cat$ \emph{defines its Beth companion}, in the sense that $\B{\cat}$ can be axiomatised by a specific type of quotient of the theory of $\cat$ (Theorem~\ref{t:syntactic-Beth-companion}). 
Note that if $\cat$ satisfies \EAP, then $\mathbb{T}_{\X}$ is strongly finitary if, and only if, the unique (up to isomorphism) saturated epi-extension of any finitely presentable object of $\cat$ is also finitely presentable (cf.\ Remark~\ref{rmk:EAP-unique-saturated-epi-ext}).

\begin{remark}
The requirement for a monad to be strongly finitary is fairly restrictive. For instance, a monad on $\Set$ is strongly finitary if, and only if, its associated category of algebras is a finitary variety that is locally finite in the sense of universal algebra, meaning that any free algebra on a finite set is finite. 
Nevertheless, the monad $\mathbb{T}_\mathcal{X}$ is strongly finitary in several cases, including those of bounded distributive lattices and cancellative commutative monoids (cf.\ Section~\ref{s:examples}).
\end{remark}

Let us begin with the following observation, which describes when the Beth companion exists, provided that $\cat$ satisfies \TEM~and $\mathbb{T}_\mathcal{X}$ is strongly finitary, and establishes a link with $\omega$-orthogonality classes.

\begin{theorem} \label{thm:orthinverter}
Let $\cat$ be an lfp category satisfying \TEM, and suppose the monad~$\mathbb{T}_{\X}$ is strongly finitary. 
The following statements are equivalent:
\begin{enumerate}[label=(\arabic*)]
\item\label{i:TEM-strongly-fin-has-BC} $\cat$ admits a Beth companion;
\item\label{i:TEM-set-of-fp-epimonos} there exists a set $\ar{H}$ of epi-monos between finitely presentable objects of $\cat$ such that $\ort{\ar{H}}=\ort{\X}$;
\item\label{i:AP-Sat-ort-class} $\ort{\X}$ is an $\omega$-orthogonality class in $\cat$. 
\end{enumerate}
\end{theorem}

\begin{proof}
 Note that for the implications \ref{i:TEM-strongly-fin-has-BC} $\IMP$ \ref{i:TEM-set-of-fp-epimonos} $\IMP$ \ref{i:AP-Sat-ort-class} it would suffice to assume that $\cat$ satisfies \EAP, instead of \TEM~(see Remark~\ref{rmk:EAP-vs-TEM}).
 
\ref{i:TEM-strongly-fin-has-BC} $\IMP$ \ref{i:TEM-set-of-fp-epimonos}. In view of Theorem~\ref{t:Beth-companion-amalgam-orthogonal} and Lemma~\ref{l:X-ort-dir-colim-fp-domains}, $\ort{\X}=\X^{\perp}_{\mathrm{fpd}}$, where $\X_{\mathrm{fpd}}$ is the class of epi-monos in $\cat$ with finitely presentable domains. In fact, the proof of the latter lemma shows that $\ort{\X}=\ort{\ar{H}}$, where $\ar{H}$ consists of the components, at finitely presentable objects, of the unit $\eta$ of the adjunction formed by the inclusion $\ort{\X}\into \cat$ and its left adjoint. Since the full subcategory of $\cat$ consisting of the finitely presentable objects is essentially small, we can assume that $\ar{H}$ is a set. Because $\mathbb{T}_{\X}$ is strongly finitary, $\ar{H}$ consists of epi-monos between finitely presentable objects.

\ref{i:TEM-set-of-fp-epimonos} $\IMP$ \ref{i:AP-Sat-ort-class}. Clear.

\ref{i:AP-Sat-ort-class} $\IMP$ \ref{i:TEM-strongly-fin-has-BC}. Since every $\omega$-orthogonality class in $\cat$ is also an $\omega$-injectivity class, Corollary~\ref{cor:Beth-TEM-omega-ort} implies that $\cat$ admits a Beth companion.
\end{proof}

Next, we shall review the basic elements of Gabriel--Ulmer duality. For the rest of this section, we will assume that the reader is familiar with the basic concepts of $2$-category theory. See e.g.\ \cite{Lack2010}.

\begin{definition}
We define the following $2$-categories:
\begin{itemize}
\item $\Lfp$ has lfp categories as objects, lfp morphisms (see Section~\ref{s:lfp}) as arrows, and natural transformations between lfp morphisms as $2$-cells.
\item $\Lex$ has small categories with finite limits as objects, functors that preserve finite limits as arrows, and natural transformations between such functors as $2$-cells.
\end{itemize}
\end{definition}

Gabriel--Ulmer duality establishes a dual (bi)equivalence between these $2$-categories:

\begin{theorem}[Gabriel--Ulmer duality]\label{t:GU-duality}
There exists a biequivalence of $2$-categories  
\[\Lex(-,\Set)\colon \Lex^{\op} \leftrightarrows \Lfp \cocolon\Lfp(-,\Set).\]
\end{theorem}

This duality was first established by Gabriel and Ulmer~\cite{GabrielUlmer1971}. Later expositions include~\cite{ar94book}. For a presentation that emphasises the $2$-dimensional aspects, we refer interested readers to~\cite{MakkaiPitts1987,AP1998}; see also \cite{LackPower2009, Tendas2025}.
Historically, the object of $\Lex$ corresponding to an lfp category $\cat$ was presented as the opposite of the full subcategory $\fp{\cat}$ of $\cat$ consisting of the finitely presentable objects. The following observation is due to Makkai and Pitts \cite[Corollary~1.5]{MakkaiPitts1987}.

\begin{proposition}\label{prop:twofaces} 
The category $\fp{\cat}$ is dually equivalent to $\Lfp(\cat,\Set)$.
\end{proposition}

Gabriel--Ulmer duality can be regarded as a \emph{syntax-semantics} duality. In fact, categories with finite limits are precisely the syntactic categories of finitary essentially algebraic theories, and lfp categories are the categories of models of such theories (cf.\ Section~\ref{s:lfp}). Accordingly, the $2$-functor $\Lex(-,\Set)$ can be identified with the $2$-functor $\mathsf{Mod}$ sending a theory to its category of models in $\Set$, and $\Lfp(-,\Set)$ with the $2$-functor $\mathsf{Th}$ sending a category of models to its theory.

\begin{remark}
Essentially algebraic logic has many different presentations. It was first introduced by Freyd under the name of \emph{Cartesian logic} \cite{Freyd1972AspectsOfTopoi,Freyd2002CartesianLogic}. Over the years, alternative approaches emphasising different syntactic elements have emerged. These include Cartmell's generalised algebraic theories \cite{Cartmell1986GeneralisedAlgebraicTheories}, and Palmgren and Vickers' partial Horn theories \cite{PalmgrenVickers2007PartialHorn}. 
These different presentations of the syntax are equiexpressive, and the corresponding syntactic categories are precisely the categories with finite limits. There are other categorical models of essentially algebraic theories. For example, \emph{clans} provide a more precise representation of generalised algebraic theories \cite{Frey2025DualityClans,ALN25}, while \emph{discrete Cartesian restriction categories} capture the partiality of some operations more faithfully \cite{DiLibertiEtAl2021PartialTheories}. 
\end{remark}

We will now lay the groundwork for axiomatising Beth companions. Note that the data of a set $\ar{M}$ of morphisms between finitely presentable objects in $\cat$ can be encoded as a natural transformation $[\ar{M}]$ as on the left-hand side below, where $\overline{\ar{M}}$ is the small full subcategory of the arrow category of $\cat$ consisting of the arrows in $\ar{M}$.
\[
\begin{tikzcd}
	\overline{\mathcal{M}}
	&& {\cat_\omega}
	&& {\makebox[0.7cm][c]{$\cat$}}
	&& {\makebox[0.7cm][c]{$[\overline{\mathcal{M}}^{\op},\Set]$}}
	\arrow[""{name=0, anchor=center, inner sep=0},
		"{\mathsf{dom}(\mathcal{M})}",
		bend left=40, from=1-1, to=1-3]
	\arrow[""{name=1, anchor=center, inner sep=0},
		"{\mathsf{cod}(\mathcal{M})}"',
		bend right=40, from=1-1, to=1-3]
	\arrow[""{name=2, anchor=center, inner sep=0},
		"{\mathsf{dom}(\mathcal{M})^*}",
		bend left=40, from=1-5, to=1-7]
	\arrow[""{name=3, anchor=center, inner sep=0},
		"{\mathsf{cod}(\mathcal{M})^*}"',
		bend right=40, from=1-5, to=1-7,shorten >=3pt]
	\arrow["{[\mathcal{M}]}"{description}, Rightarrow,
		from=0, to=1, shorten <=2pt, shorten >=2pt]
	\arrow["{[\mathcal{M}]^*}"{description}, Rightarrow,
		from=3, to=2, shorten <=2pt, shorten >=2pt]
\end{tikzcd}
\]

It follows from Theorem~\ref{t:GU-duality} and Proposition~\ref{prop:twofaces} that such a representation has a specular description in $\Lfp$ given by the rightmost diagram above. Just observe that, under Gabriel--Ulmer duality, the presheaf category $[\overline{\ar{M}}^\op,\Set]$ is dual to the finite limit completion of~$\overline{\ar{M}}^\op$, and $\cat$ is dual to $(\cat_\omega)^\op$. 
Explicitly, the natural transformation $[\ar{M}]^*$ is given by the following formula, for every object $X$ of $\cat$ and every arrow $m\colon A \to B$ in $\ar{M}$:
\[
([\ar{M}]^*)_X(m) = \cat(B,X) \to \cat(A,X) \quad f \mapsto f \circ m.
\]

As we will see in Proposition~\ref{p:Beth-inverter} below, under appropriate assumptions on $\cat$, its Beth companion has a universal property with respect to a diagram such as the one on the right above. To this end, we recall the notion of an inverter in any $2$-category, which generalises the idea of an equaliser.
Given a $2$-cell $\alpha \colon f\Rightarrow g$, the \emph{inverter} of~$\alpha$, if it exists, consists of an object $\mathsf{Inv}(\alpha)$ together with a morphism $j$, as displayed below, such that (a)~$\alpha_j$ is an isomorphism and (b)~it is universal among the objects satisfying (a), as detailed in \cite[\S 4, p.~310]{Kelly1989}.

\[\begin{tikzcd}
	{\mathsf{Inv}(\alpha)} & \phantom{\cdot} & & \phantom{\cdot}
	\arrow["j"', dashed, swap, from=1-1, to=1-2]
	\arrow[""{name=0, anchor=center, inner sep=0}, "g"', curve={height=18pt}, from=1-2, to=1-4]
	\arrow[""{name=1, anchor=center, inner sep=0}, "f", curve={height=-18pt}, from=1-2, to=1-4]
	\arrow["\alpha"', swap, Rightarrow, from=1, to=0, shorten <= 4pt, shorten >= 4pt]
\end{tikzcd}\]
If $\alpha$ is of the form $[\ar{M}]^*$, we say that $\mathsf{Inv}(\alpha)$ is the \emph{inverter of $\ar{M}$}.

In $\CAT$, the $2$-category of locally small categories and functors between them, the inverter $\mathsf{Inv}(\alpha)$ coincides with the full subcategory of those objects $X$ such that $\alpha_{X}$ is an isomorphism.

\begin{proposition}\label{p:Beth-inverter}
  Let $\cat$ be an lfp category satisfying \EAP, and suppose the monad~$\mathbb{T}_{\X}$ is strongly finitary. If $\cat$ admits a Beth companion, then the latter is the inverter in $\Lfp$ of a set of epi-monos between finitely presentable objects of~$\cat$.
\end{proposition}

\begin{proof}
It follows from Remark~\ref{rmk:EAP-vs-TEM} and Theorem~\ref{thm:orthinverter} that there exists a set $\ar{H}$ of epi-monos between finitely presentable objects of $\cat$ such that the Beth companion of $\cat$ is the orthogonality class of $\ar{H}$.
We must show that the following diagram is an inverter in $\Lfp$.
\[
  \begin{tikzcd}
    {\mathcal{H}^\perp} &
    {\makebox[0.5cm][c]{$\cat$}}
    && {\makebox[0.5cm][c]{$\Set^{\overline{\mathcal{H}}}$}}
    \arrow[hook, from=1-1, to=1-2]
    \arrow[""{name=2, anchor=center, inner sep=0},
      "{\mathsf{dom}(\mathcal{H})^*}",
      bend left=40, from=1-2, to=1-4]
    \arrow[""{name=3, anchor=center, inner sep=0},
      "{\mathsf{cod}(\mathcal{H})^*}"',
      bend right=40, from=1-2, to=1-4, shorten >=3pt]
    \arrow["{[\mathcal{H}]^*}"{description}, Rightarrow,
      from=3, to=2, shorten <=2pt, shorten >=2pt]
  \end{tikzcd}
\]

To this end, observe that the forgetful functor $\Lfp \to \CAT$ creates all $2$-limits, and in particular inverters, because it is $2$-conservative and preserves $2$-limits; for the latter fact, see \cite[Theorem~2.17]{Bird1984}. In turn, as mentioned above, the inverter in $\CAT$ is the full subcategory of $\cat$ consisting of the objects $X$ such that, for each arrow $h\colon A \to B$ in $\ar{H}$, the induced map $\cat(B,X) \to \cat(A,X)$, $f \mapsto f \circ h$, is a bijection. This means precisely that $\ort{\ar{H}}$ is the desired inverter. 
\end{proof}

We can now move across the duality to infer a syntactic presentation of Beth companions. The notion of \emph{coinverter} in a $2$-category is dual to that of inverter. In particular, the coinverter of a $2$-cell of the form $[\ar{M}]$ is called the \emph{coinverter of $\ar{M}$}. The following result is an immediate consequence of Proposition~\ref{p:Beth-inverter}.

\begin{theorem}\label{t:syntactic-Beth-companion}
  Let $\cat$ be an lfp category satisfying \EAP, and suppose the monad~$\mathbb{T}_{\X}$ is strongly finitary. If the Beth companion $\B{\cat}$ exists, its theory $\mathsf{Th}(\B{\cat})$ is the coinverter in $\Lex$ of a set of epi-monos in $\mathsf{Th}(\cat)$.
\end{theorem}

\begin{remark}
  It is worth noting that the latter theorem is far from trivial, as the forgetful functor $\Lex \to \Cat$ does not preserve coinverters, and the description of coinverters in $\Lex$ is almost entirely formal. Indeed, $\Lex$ is $2$-monadic over $\Cat$, and colimits in categories of algebras are seldom computed in the underlying category, especially coequaliser-type colimits. Our proof relies crucially on Gabriel--Ulmer duality and on the fact that $2$-limits in $\Lfp$ are computed in $\CAT$.
\end{remark}

Under the assumptions of Theorem~\ref{t:syntactic-Beth-companion}, the resulting quotient $\mathsf{Th}(\cat) \to \mathsf{Th}(\B{\cat})$ is a pseudo-epimorphism in $\Lex$ \cite[Lemma~2.29]{Bourke2010Codescent}, and it can therefore be understood as a \emph{quotient theory}. In \cite[\S III.1]{HebertAdamekRosicky2001}, these pseudo-epimorphisms are called \emph{quotient theory morphisms}, or \emph{quotient functors}, and it is shown that they are surjective in an appropriate sense. Such quotients only add axioms to the theory $\mathsf{Th}(\cat)$, as opposed to new operations or sorts; cf.\ \cite[p.~486]{MakkaiPitts1987}.

\section{Examples}\label{s:examples}

\subsubsection*{Posets} Let $\Pos$ be the lfp category of posets and monotone maps. For every poset~$P$ there exists a monotone bijection $P\to Q$ (i.e., an epi-mono in $\Pos$) with $Q$ totally ordered~\cite{Szpilrajn1930}. Thus, the saturated objects in $\Pos$ are precisely the totally ordered ones. But the full subcategory $\Tot$ of $\Pos$ consisting of the totally ordered posets is not reflective, e.g.\ because the inclusion $\Tot\into \Pos$ does not preserve products. Therefore, it follows from Theorem~\ref{t:Beth-companion-unique} that $\Pos$ does not admit a Beth companion.

\subsubsection*{Preorders} The lfp category of preordered sets and monotone maps admits a strong Beth companion, which coincides with the full subcategory formed by the indiscrete preorders (meaning that the preorder relation is total).

\subsubsection*{Bounded distributive lattices}

The variety $\DL$ of bounded distributive lattices has transferable monos; cf.\ e.g.\ the recent survey~\cite{Metcalfe2026}. Let $i$ be the inclusion of the free bounded distributive lattice on one generator into the free Boolean algebra on one generator (regarded as a lattice), which sends the generator~$x$ to itself, as depicted below.

\[\begin{tikzpicture}[
    scale=0.75,
    node distance=1.5cm,
    element/.style={circle, inner sep=2pt},
    inclusion/.style={
      dashed, 
      thick, 
      shorten >=6pt, 
      shorten <=6pt, 
      color=gray,
      -{Stealth[length=2.5mm]}
    }
  ]
  
  \node[element] (0D) at (0,0) {$0$};
  \node[element] (xD) at (0,1) {$x$};
  \node[element] (1D) at (0,2) {$1$};
  
  \draw[thick] (0D) -- (xD) -- (1D);

  \node[element] (0B) at (4,0) {$0$};
  \node[element] (xB) at (3,1) {$x$};
  \node[element] (nxB) at (5,1) {$\neg x$};
  \node[element] (1B) at (4,2) {$1$};
  
  \draw[thick] (0B) -- (xB) -- (1B);
  \draw[thick] (0B) -- (nxB) -- (1B);

\draw[{[right,length=1.0mm]}->,shorten >=10pt, shorten <=10pt] (xD) -- node[above] {$i$} (xB);
\end{tikzpicture}\]
Note that $i$ is epi-monic in $\DL$. It is not difficult to see that the orthogonality class~$\ort{i}$ coincides with $\BA$, the full subcategory of $\DL$ consisting of Boolean algebras. Since $\BA$ is balanced, Corollary~\ref{cor:S-perp-Beth-comp} implies that $\BA$ is the strong Beth companion of $\DL$ (recall from Section~\ref{s:preliminaries-amalgamation} that, in any lfp category $\cat$ with the amalgamation property, $\Sat(\cat)=\Abs(\cat)$).

\subsubsection*{Commutative monoids} 
The variety $\CMon$ of commutative monoids does not satisfy the amalgamation property (cf.\ again~\cite{Metcalfe2026}). We shall prove that $\CMon$ does not admit a Beth companion, let alone a strong one. In fact, a stronger result holds: there exists no balanced full mono-reflective subcategory of~$\CMon$.

For every non-negative integer $n$, consider the commutative monoid $E_{n}$ with generators $a,b$ and relations
\[
2b + a = b, \ \ b + (n+1) a = n a,
\]
where $kx$, for $x\in \{a,b\}$, denotes the sum of $x$ with itself $k$ times. Observe that
\[
b + n a = (n-1) a \ \Longrightarrow b + (n+1) a = n a,
\]
and so there is a natural quotient map $E_{n+1}\twoheadrightarrow E_{n}$. Further, $E_{0}$ is isomorphic to~$\Z$, the monoid of integers. Hence, we have a chain of quotients as displayed below.
\[
\cdots \twoheadrightarrow E_{n+1} \twoheadrightarrow E_{n}\twoheadrightarrow E_{n-1} \twoheadrightarrow \cdots \twoheadrightarrow \Z
\]
Each $E_{n}$ is an epi-extension of $\N$, and they are the only ones besides the identity (cf.\ \cite[p.~241]{Isbell1966}). In particular, they are maximal elements in the poset $\EExt(\N)$ of epi-extensions of $\N$, and so the monoids $E_{n}$ are saturated.

\begin{proposition}\label{p:CMon-no-Beth-companion}
There is no balanced full mono-reflective subcategory of~$\CMon$.
\end{proposition}
\begin{proof}
Let $\catt\into \CMon$ be a balanced full mono-reflective subcategory, with unit $\eta$. We can assume without loss of generality that $\catt$ is isomorphism-closed in $\CMon$. By Lemma~\ref{l:mono-refl-sat}, $\catt$ contains all saturated monoids; in particular, it contains the monoids~$E_{n}$ defined above for every non-negative integer $n$, including $\Z$. 

Consider the component~$\eta_{\N}$ of the unit at $\N$. If $\eta_{\N}$ is an isomorphism, then $\N$ belongs to $\catt$. But then the epi-monic inclusion $i\colon \N\into \Z$ belongs to $\catt$, contradicting the fact that $\catt$ is balanced.
On the other hand, if $\eta_{\N}$ is a proper epi-extension, it can be identified with the arrow ${\iota_{n}\colon \N\to E_{n}}$ that sends $1$ to $a$, for some $n$. Now, consider the unique arrow $\xi$ making the following diagram commute.
\[\begin{tikzcd}
\N \arrow{r}{\eta_{\N}} \arrow{dr}[swap]{\iota_{n+1}} & E_{n} \arrow{d}{\xi} \\
{} & E_{n+1}
\end{tikzcd}\]
If $q_{n}\colon E_{n+1}\twoheadrightarrow E_{n}$ is the natural quotient, the universal property of $\eta_{\N}$ implies that $q_{n}\circ \xi =\id_{E_{n}}$. This shows that $\xi$ is both a section and an epimorphism, hence an isomorphism. This is a contradiction.
\end{proof}

Note that $\CMon$ does not satisfy \TEM. Just observe that the orthogonality class~$\ort{i}$ in $\CMon$ induced by the epi-monic inclusion $i\colon \N\into\Z$ is the full subcategory $\AGrp$ of $\CMon$ formed by Abelian groups. Hence $\AGrp \subseteq \ort{\X}$, where $\X$ is the class of epi-monos in $\CMon$. The converse inclusion holds because the Grothendieck group functor $\CMon\to \AGrp$ inverts all epi-monos, so $\AGrp=\ort{\X}$. On the other hand, it follows from the above discussion that not every saturated object in $\CMon$ is a group. Therefore, $\ort{\X}$ is a proper subclass of $\Sat(\CMon)$. It follows from Lemma~\ref{l:ort-vs-sat} that $\CMon$ does not have transferable epi-monos.

\subsubsection*{Cancellative commutative monoids} The sub-quasivariety of $\CMon$ consisting of cancellative commutative monoids admits a strong Beth companion, namely the variety of Abelian groups (which satisfies \SES, as do all Abelian categories). The reflector sends a cancellative commutative monoid to its group of differences.

\subsubsection*{Cancellative partially ordered commutative monoids.} 
A \emph{partially ordered commutative monoid} is a pair $(M,\leq)$ where $M$ is a commutative monoid and ${\leq}\subseteq M\times M$ is a translation-invariant partial order, meaning that $x\leq y$ implies $x+ t\leq y+t$ for all $x,y,t\in M$. A partially ordered commutative monoid is \emph{cancellative} if it is cancellative as a monoid. Let~$\COAMon$ be the category of cancellative partially ordered commutative monoids, with morphisms the monotone monoid homomorphisms. This is a finitary Horn class, hence lfp. We claim that $\COAMon$ does not admit a Beth companion.

Firstly, note that an arrow $f\colon X \to Y$ in $\COAMon$ whose domain is a group is epi-monic if, and only if, it is a monotone monoid isomorphism. For the non-trivial direction, suppose that $f$ is epi-monic in $\COAMon$. If $Y$ is a group, then $f$ is also epi-monic in the full subcategory $\OAGrp$ of $\COAMon$ consisting of partially ordered Abelian groups; in turn, epi-monos in $\OAGrp$ are precisely the monotone group isomorphisms.  
Hence, $f$ is a monoid isomorphism. For the general case, observe that $\OAGrp$ is a mono-reflective subcategory of~$\COAMon$, whose reflector extends the order on an object of~$\COAMon$ to a translation-invariant order on its group of differences in the obvious way. Hence, $Y$ admits an epi-mono $g$ to a partially ordered Abelian group $Z$. By the previous argument, $g\circ f$ is a monoid isomorphism, and so $f$ is a monoid isomorphism.

Therefore, every totally ordered Abelian group is a saturated object of $\COAMon$. Moreover, as $\OAGrp$ is a mono-reflective subcategory of $\COAMon$, Lemma~\ref{l:mono-refl-sat} implies that every saturated object of $\COAMon$ is a group. Assume that $\COAMon$ admits a balanced, full and replete mono-reflective subcategory, which then coincides with the category of saturated objects of $\COAMon$ by Proposition~\ref{p:saturated-Beth-comp-lfp}. Consider the reflection $s\colon (\Z,=)\to S$ of the group of integers with the trivial order, and let $\sqsubseteq$ be any translation-invariant total order on $\Z$. As $(\Z, \sqsubseteq)$ is saturated, there is a unique arrow $t\colon S \to (\Z,\sqsubseteq)$ that extends the identity map $i\colon (\Z,=) \to (\Z,\sqsubseteq)$. Both $s$ and $i$ are epi-monos in~$\COAMon$. Since $S$ is a group, and epi-monos between groups in~$\COAMon$ have the 2-out-of-3 property because they coincide with monotone monoid isomorphisms, $t$ is epi-monic. As $S$ is saturated, $t$ must be an iso. Upon recalling that there are (exactly) two distinct translation-invariant total orders on~$\Z$, namely the usual order $\leq$ and its opposite, the universal property of the reflector entails the existence of an isomorphism $(\Z,\leq)\to (\Z,\leq^{\op})$ whose underlying group homomorphism is the identity, a contradiction.

This implies that $\COAMon$ admits no balanced mono-reflective subcategory. In particular, it does not admit a Beth companion, let alone a strong one.

\medskip
The remaining examples concern classes of rings, which we will always assume to be commutative and unital. 

\subsubsection*{Commutative rings.} 
The category $\CRng$ of commutative unital rings (simply \emph{rings}, for short) and unital ring homomorphisms does not admit a Beth companion. This can be deduced by looking at the (saturated) epi-extensions of $\Z$, the initial object of $\CRng$. The rings $R$ such that the unique arrow $\Z \to R$ is an epimorphism were studied by Bousfield and Kan under the name of \emph{solid rings}~\cite{BK1972}. For any ring~$R$, denote by $R^t$ its torsion subgroup. According to the characterisation established in \emph{op.\ cit.}, a solid ring $R$ is of one of the following four types:
\begin{enumerate}[label=(\Roman*)]
\item\label{i:type-1} If $R$ is a torsion group (i.e., $R^t=R$) or, equivalently, it has positive characteristic, then $R$ is isomorphic to a cyclic ring~$\Z_n$ for some $n> 1$.
\item\label{i:type-2} If $R$ is torsion-free (i.e., $R^t=\{0\}$), then $R$ is isomorphic to a subring of $\Q$; that is, to a ring $\Z[J^{-1}]$, for some set of primes $J$, consisting of the rationals whose denominators only involve primes in $J$. 
\item\label{i:type-3} If $R$ has mixed bounded torsion (i.e., $\{0\}\subsetneq R^t \subsetneq R$ and $m R^t=\{0\}$ for some $m> 1$), then $R$ is isomorphic to a ring of the form $\Z[J^{-1}] \times \Z_n$, where each prime factor of $n$ belongs to $J$.
\item\label{i:type-4} If $R$ has mixed unbounded torsion, then it is a \emph{large solid ring}; these can be characterised as the rings that are not of type~\ref{i:type-1}--\ref{i:type-3}, but can be written as a directed colimit of such rings.
\end{enumerate}
Epi-extensions of $\Z$ correspond precisely to the solid rings of characteristic $0$, i.e., those of type~\ref{i:type-2}--\ref{i:type-4}.

\begin{proposition}\label{p:CRng-no-Beth-companion}
  The category $\CRng$ of commutative unital rings does not admit a Beth companion.
\end{proposition}
\begin{proof}
Suppose that $\CRng$ has a Beth companion. In this case, it coincides with its full subcategory of saturated rings by Theorem~\ref{t:Beth-companion-unique}. In view of \cite[Korollar~5.4, Korollar~6.3]{Storrer1968}, saturated rings include all fields and all finite cyclic rings~$\Z_n$.

Let $\eta_{\Z}\colon \Z\to R$ be the component at $\Z$ of the unit of the adjunction formed by the inclusion $\Sat(\CRng)\into \CRng$ and its left adjoint. Since $\Z$ is not saturated, the saturated ring $R$ is a proper epi-extension of $\Z$. We distinguish two cases, depending on whether $R$ has torsion.

If $R$ is torsion-free, then it is a solid ring of type~\ref{i:type-2}. But every subring of $\Q$ admits an epi-mono to $\Q$, so the only saturated solid ring of type~\ref{i:type-2} is $\Q$. The universal property of $R$ then implies that, for each prime $p$, there is an arrow $\Q\to \mathbb{F}_{p}$, which is a contradiction.

If $R$ has torsion, by \cite[Lemma~3.10]{BK1972} we have 
\begin{equation}\label{eq:solid-torsion}
  R^t\cong \bigoplus_{p\in K}{\Z_{p^{e(p)}}} \ \ \text{ and } \ \  R/R^t \cong \Z[J^{-1}]
\end{equation}
for some non-empty sets of primes $K\subseteq J$, and positive integers $e(p)$.\footnote{If $K$ is finite then $R$ is of type~\ref{i:type-3}; otherwise, $R$ is of type~\ref{i:type-4}.} Because $R$ is saturated, $J$ must be the set of all primes (just note that, if $q\notin J$, then $R\to R[q^{-1}]$ is a proper epi-mono), and so $R/R^t \cong \Q$. Consider any prime $p\in K$ and set $k\coloneqq e(p)+1$. Then the universal property of $R$ yields an arrow $f\colon R\to \Z_{p^k}$. We will show that this leads to a contradiction. Since $R/R^t \cong \Q$, $p$ is invertible modulo the torsion subgroup; that is, there are $x\in R$ and $t\in R^t$ such that $px=1+t$. The description of $R^t$ in eq.~\eqref{eq:solid-torsion} implies the existence of a finite set $\{q_1,\ldots,q_n\}\subseteq K \setminus\{p\}$ such that
\[
t = t_p + t_{q_1} + \dots + t_{q_n},
\]
with $t_p \in \Z_{p^{e(p)}}$ and $t_q \in \Z_{q^{e(q)}}$ for every $q\in \{q_1,\ldots,q_n\}$. Each such prime $q$ is invertible in $\Z_{p^k}$, and each $t_q$ is annihilated by some power of $q$, so $f(t_q)=0$. Hence, $f(t)= f(t_p)$. In particular,
\[
  p^{e(p)}\cdot f(t) = p^{e(p)}\cdot f(t_p) = f(p^{e(p)}t_p) = f(0) = 0.
\]
Therefore, the identity $px=1+t$ yields
\[
0 = p^k\cdot f(x) = p^{e(p)}(p\cdot f(x)) = p^{e(p)}(1 + f(t)) = p^{e(p)},
\]
which is false in $\Z_{p^k}$.
\end{proof}

\medskip
To motivate the following examples, let us first consider the category $\Field$ of fields, regarded as a full subcategory of $\CRng$. As $\Field$ is neither complete nor cocomplete (for instance, it lacks initial and terminal objects), it is not lfp. Also, $\Field$ is not balanced: for every prime number $p$, let $\mathbb{F}_{p}$ be the finite field with $p$ elements, and let $\mathbb{F}_{p}(t)$ be the field of rational functions with coefficients in $\mathbb{F}_{p}$. Then the Frobenius endomorphism
\[
\mathbb{F}_{p}(t) \to \mathbb{F}_{p}(t), \ \ x\mapsto x^{p}
\]
is epi-monic in $\Field$, but not an isomorphism. The orthogonality class in $\Field$ induced by the Frobenius endomorphisms on $\mathbb{F}_{p}(t)$, for each prime $p$, is the class of \emph{perfect fields}. Recall that a field is perfect if it has either characteristic $0$, or it has positive characteristic $p$ and its Frobenius endomorphism is an isomorphism (meaning that each element has a necessarily unique $p$-root). The full subcategory $\Field_{\pi}$ of $\Field$ consisting of perfect fields is mono-reflective: the reflector sends a field $F$ of characteristic $p$ to its \emph{perfect closure} $F^{1/p^{\infty}}$, and is the identity on fields of characteristic $0$. We record the following fact for future reference.
\begin{lemma}\label{l:perfect-fields-balanced}
The category $\Field_{\pi}$ of perfect fields is balanced. 
\end{lemma}
\begin{proof}
Every ring homomorphism between fields is injective. By Lemma~\ref{l:basic-props-mono-reflections}, epimorphisms in the category of perfect fields coincide with epimorphisms in the category of fields. Therefore, it suffices to show that every epimorphism in $\Field$ between perfect fields is surjective. In turn, it is well known that epimorphisms in $\Field$ are precisely the purely inseparable field extensions. Since perfect fields do not admit any non-trivial purely inseparable extension, every epimorphism in $\Field$ whose domain is a perfect field must be surjective.
\end{proof}

Perfect fields are saturated by the proof of Lemma~\ref{l:perfect-fields-balanced}; since $\Field_{\pi}$ is mono-reflective in $\Field$, Lemma~\ref{l:mono-refl-sat} yields that every saturated field is perfect. Thus, perfect fields are precisely the saturated objects in $\Field$, a fact already observed in~\cite[Beispiel~12.2(a)]{Storrer1968}.

To import these ideas into the lfp setting, recall that under Pierce's sheaf representation of rings, a ring $R$ is isomorphic to the ring of global sections of a sheaf of fields (on a Stone space) precisely when $R$ is a \emph{(von Neumann) regular ring}, meaning that $a\in aRa$ for every $a\in R$; see \cite[Theorem~10.3]{Pierce1967} or \cite[Proposition~V.2.6]{Johnstone1982}. Non-necessarily commutative regular rings were introduced in~\cite{vonNeumann1936}; see~\cite{Goodearl1991} for a complete account. Every regular ring is \emph{reduced}, i.e.\ it has no nilpotent elements. 
For a regular ring $R$, its Pierce spectrum is homeomorphic to its maximal spectrum. The Pierce sheaf representation of $R$, at a maximal ideal $m$ of $R$, has the field $R/m$ as its stalk. Let us say that a regular ring $R$ is \emph{perfect} if, for every maximal ideal $m$ of $R$, the field $R/m$ is perfect.\footnote{The term ``perfect ring'' is sometimes used in the literature to refer to rings with either characteristic zero or a positive characteristic for which the Frobenius endomorphism is an isomorphism. The latter definition is not comparable with ours.} 

As we shall see below, perfect regular rings are the strong Beth companion of both regular rings and reduced rings. This was proved using algebraic tools in~\cite{CKM2026rings}; the observation that perfect regular rings form a mono-reflective subcategory of reduced rings can already be found in~\cite{BKR2015}.\footnote{Perfect regular rings are equivalent to the ``implicitly closed meadows'' of~\cite{CKM2026rings}.} Our proofs differ substantially in their use of orthogonality classes.

\subsubsection*{Regular rings}
Let $\RegRng$ be the category of regular rings, viewed as a full subcategory of the category $\CRng$ of rings and ring homomorphisms. It is well-known that a ring $R$ is regular precisely when each $a\in R$ admits a \emph{pseudo-inverse}, i.e.\ an element $a^{*}$ such that $a a^{*} a = a$ and $a^{*} a a^{*} = a^{*}$. If they exist, pseudo-inverses are unique. In fact, $\RegRng$ is isomorphic to the variety of algebras obtained by adding a pseudo-inverse operation $(-)^{*}$ to the language of rings, and the equations just mentioned to the axioms for rings; see e.g.\ \cite[pp.~428--429]{Cornish1977}. In particular, $\RegRng$ is lfp. However, it is not balanced: for example, as fields are regular rings, the Frobenius endomorphisms $\mathbb{F}_{p}(t) \to \mathbb{F}_{p}(t)$ are epi-monic in $\RegRng$, but not isomorphisms.
We claim that the full subcategory $\RegRng_{\pi}$ of $\RegRng$ consisting of perfect regular rings is the strong Beth companion of $\RegRng$. 

To begin with, we will prove that a regular ring $R$ is perfect if, and only if, it satisfies the condition
\begin{equation}\label{eq:p-axiom}
px = 0 \ \IMP \ \exists y \, (y^{p}=x) \tag{$E_{p}$}
\end{equation}
for every prime $p$ and every $x\in R$.
In a reduced ring, a $p$-root of an element annihilated by $p$ is unique when it exists. Hence, the formulas~\eqref{eq:p-axiom} are Cartesian relative to the theory of reduced rings. Recall that Cartesian formulas are preserved by the global section functor; see e.g.\ \cite[\S V.1.12]{Johnstone1982}. As~\eqref{eq:p-axiom} is easily seen to hold in all perfect fields, it holds in all perfect regular rings (just recall that every perfect regular ring is the ring of global sections of a sheaf of perfect fields).\footnote{In fact, since the Pierce spectrum is a Stone space, the global section functor preserves all regular formulas (see e.g.\ \cite[\S V.1.13]{Johnstone1982}), meaning that we do not need to restrict ourselves to provably unique existential quantifiers.} Conversely, suppose $R$ is a regular ring satisfying~\eqref{eq:p-axiom} for all primes~$p$. We must show that, for every maximal ideal $m$ of~$R$, the field $R/m$ is perfect. If $R/m$ has characteristic $0$, there is nothing to prove. Hence, suppose $R/m$ has positive characteristic $p$. We aim to prove that each $x\in R/m$ has a $p$-root. To this end, we note the following.
\begin{lemma}\label{l:lifting-p-annihilation}
In any regular ring $A$, consider the element $e\coloneqq pp^{*}$, where $p$ denotes the sum of the unit $1$ with itself, $p$ times. For every $x\in A$, we have $p(1-e)x=0$. 
\end{lemma}

\begin{proof}
Computing the term $p(1-e)x$, we get
\[
p(1-e)x = (p-pe)x = (p-pp^{*}p)x=(p-p)x=0\cdot x =0.\qedhere
\]
\end{proof}

If $x\in R/m$, the element $x'\coloneqq (1-e)x$ satisfies $px'=0$ by Lemma~\ref{l:lifting-p-annihilation}. By~\eqref{eq:p-axiom}, there is $y\in R$ such that $y^{p}=x'$. Note that $p\in m$ since $R/m$ has characteristic $p$; hence $e\in m$ as the latter is an ideal. We get $x'\equiv (1-0)x \equiv x \pmod{m}$, and so $y^{p}\equiv x \pmod{m}$. That is, $y$ is the $p$-root of $x$ in $R/m$.

Now, we shall translate condition~\eqref{eq:p-axiom} into an orthogonality condition. Viewing $\RegRng$ as a variety of algebras as explained above, let $U_{p}$ be the free regular ring on one generator $x$ modulo the relation $px=0$, and consider the unique morphism
\[
f_{p}\colon U_{p}\to U_{p}
\]
that sends the generator $x$ to $x^{p}$. Note that $f_{p}$ is well-defined because
\[
px^{p} = (px)x^{p-1} = 0\cdot x^{p-1}=0.
\]
A regular ring is orthogonal to $f_{p}$ if, and only if, it satisfies 
\[
px = 0 \ \IMP \ \exists y \, (y^{p}=x \text{ and } py=0).
\]
The latter condition is equivalent to~\eqref{eq:p-axiom}: just note that if $px=0$ and $y^{p}=x$, then 
\[
(py)^{p} = p^{p}y^{p} = p^{p}x = p^{p-1}\cdot 0 = 0,
\]
and so $py=0$ because every regular ring is reduced. Therefore, setting
\[
\ar{S}\coloneqq \{f_{p}\colon U_{p}\to U_{p} \mid p \text{ prime}\},
\]
we see that the orthogonality class $\ort{\ar{S}}$ in $\RegRng$ coincides with $\RegRng_{\pi}$.

We claim that each arrow $f_{p}$ is epi-monic. To see that $f_{p}$ is epic, suppose that $g,h\colon U_{p}\rightrightarrows A$ satisfy $g\circ f_{p} = h\circ f_{p}$ in $\RegRng$. It suffices to show that $g$ and $h$ coincide on the generator $x$. We have
\[
0 = g(f_{p}(x)) - h(f_{p}(x)) = g(x^{p}) - h(x^{p}) = g(x)^{p} - h(x)^{p} = (g(x)-h(x))^{p},
\]
where the last equality holds because $p\cdot g(x)=0=p\cdot h(x)$. Hence, $g(x)-h(x)=0$ because $A$ is reduced, showing that $f_{p}$ is epic. To see that $f_{p}$ is monic, i.e.\ injective, fix any $a\in U_{p}$ such that $f_{p}(a)=0$. We must show that $a=0$. Here, we can assume that $a$ is (an equivalence class of) a term in the language of regular rings built from the generator $x$, regarded as a variable. So, the condition $f_{p}(a)=0$ translates to $a(x^{p})=0$. Recall that every regular ring embeds in a product of fields (this is immediate from the Pierce sheaf representation). Therefore, to prove that $a=0$, it suffices to show that $\phi(a)=0$ for every arrow $\phi\colon U_{p}\to F$ with $F$ a field. Let $\alpha\coloneqq \phi(x)$ be the image of the generator. As $px=0$ and $a(x^{p})=0$, we get $p\alpha = 0$ and $a(\alpha^{p})=0$. We distinguish two cases:
\begin{itemize}
\item If $\alpha = 0$ then $\alpha = \alpha^{p}$, and so $\phi(a) = a(\alpha) = a(\alpha^{p}) = 0$.
\item If $\alpha \neq 0$ then, since $p\alpha = 0$, the field $F$ has characteristic $p$. In particular, the Frobenius endomorphism $(-)^{p}\colon F \to F$ is a morphism in $\RegRng$ (note that ring homomorphisms preserve all existing pseudo-inverses). Hence, $a(\alpha)^{p} = a(\alpha^{p}) = 0$. As $F$ is reduced, we get $\phi(a)=a(\alpha)=0$.
\end{itemize}

Next, we prove that $\ort{\ar{S}}=\RegRng_{\pi}$ is balanced.

\begin{proposition}
The category $\RegRng_{\pi}$ is balanced.
\end{proposition}
\begin{proof}
As $\RegRng_{\pi}$ is an $\omega$-orthogonality class in $\RegRng$ associated with a set of epi-monos, and monos are pushout stable in $\RegRng$ \cite[Theorem~1.6]{Cornish1977}, it follows that $\RegRng_{\pi}$ is mono-reflective in $\RegRng$; cf.\ \cite[\S 1.37]{ar94book}. Thus, by Lemma~\ref{l:basic-props-mono-reflections}, every epi-mono in $\RegRng_{\pi}$ is epi-monic in $\RegRng$.

Recall from \cite[Theorem~10.3]{Pierce1967} that $\RegRng$ is dually equivalent to the category of regular ringed spaces. A \emph{regular ringed space} is a pair $(X, J)$ where $X$ is a Stone space and $J$ is a sheaf of fields on $X$. A \emph{morphism of regular ringed spaces} $(f,f^{\sharp})\colon (X,J)\to (Y,K)$ consists of a continuous map $f\colon X\to Y$ and a morphism $f^{\sharp}\colon K\to f_{*}J$ of sheaves of fields over $Y$, where $f_{*}$ denotes the direct image functor induced by $f$. Under this dual equivalence, the sheaf associated to a regular ring~$R$ is defined on the Stone space dual to the Boolean algebra of idempotents of $R$, whose points can be identified with the maximal ideals of $R$. The stalk of this sheaf at a maximal ideal $m$ is the field $R/m$.

Now let $h\colon R\to S$ be epi-monic in $\RegRng_{\pi}$. By the first paragraph, $h$ is epi-monic in $\RegRng$. We can identify $R$ with an epic subalgebra of $S$, and $h$ with the inclusion $R\into S$. Let $X_{R}$ and $X_{S}$ be the Stone spaces associated with $R$ and~$S$, respectively. We claim that the continuous map $X_{S}\to X_{R}$ induced by $h$ is a homeomorphism; that is, for every maximal ideal~$m$ of $R$ there is a unique maximal ideal $n$ of~$S$ such that $n\cap R = m$. 
To see that such a maximal ideal $n$ exists, it suffices to prove that the ideal $\overline{m}$ of $S$ generated by~$m$ is proper, for then it can be extended to a maximal ideal $n$ satisfying $n\cap R \supseteq m$, and so $n\cap R = m$ by maximality of~$m$. Suppose that $1\in \overline{m}$, i.e., $1=\sum{a_{i}b_{i}}$ for finitely many elements $a_{i}\in m$ and $b_{i}\in S$. Let $I\subseteq R$ be the ideal generated by the~$a_{i}$'s. In a regular ring, every finitely generated ideal is principal and is generated by an idempotent element \cite[Theorem~1.1]{Goodearl1991}. If $e\in R$ is an idempotent element that generates $I$, then $ex=x$ for every $x\in I$. Therefore,
\[
1 =\textstyle\sum{a_{i}b_{i}} = \sum{(ea_{i})b_{i}} = e\sum{a_{i}b_{i}} = e\cdot 1 = e
\]
and so $1\in I \subseteq m$, contradicting the fact that $m$ is maximal. Hence, $\overline{m}$ is proper. 

Next, assume $n_{1},n_{2}$ are maximal ideals of $S$ satisfying $n_{1}\cap R = m = n_{2}\cap R$. 
Then $L_{1}\coloneqq S/n_{1}$ and $L_{2}\coloneqq S/n_{2}$ are field extensions of $K\coloneqq R/m$. Since $\Field$ has the amalgamation property, there is a field $F$ such that the span of field extensions $L_{1}\hookleftarrow K \into L_{2}$ can be completed to a commutative diagram as displayed below.
\[\begin{tikzcd}[row sep=1em, column sep=1em]
{} & L_{1} \arrow[hookrightarrow]{dr} & \\
K \arrow[hookrightarrow]{ur} \arrow[hookrightarrow]{dr} & & F \\
{} & L_{2} \arrow[hookrightarrow]{ur} &
\end{tikzcd}\]
Define the arrows $g_{1}\colon S\to L_{1}\into F$ and $g_{2}\colon S\to L_{2}\into F$ in $\RegRng$, where the first morphisms in the two composites are the quotient maps. 
The restrictions of $g_{1}$ and~$g_{2}$ to~$R$ coincide, however, if $n_{1}\neq n_{2}$ there exists $x\in n_{1}\setminus n_{2}$ and so $g_{1}(x) = 0$ while $g_{2}(x)\neq 0$. As the inclusion $R\into S$ is epic, we conclude that $n_{1}=n_{2}$. 

For each maximal ideal $n$ of $S$, write $m\coloneqq n\cap R$. The induced field homomorphism $R/m\to S/n$ is monic. It is also epic in $\Field$: if $u,v\colon S/n\rightrightarrows F$ agree on $R/m$, then the composites $S\to S/n\rightrightarrows F$ agree on $R$, and therefore coincide because $R\into S$ is epic. Thus, each stalk map of the Pierce sheaves is epi-monic in $\Field$. Since $R$ and~$S$ are perfect, each residue field is perfect, so Lemma~\ref{l:perfect-fields-balanced} yields that every stalk map is an isomorphism. Identifying $X_{S}$ with $X_{R}$ via the homeomorphism established above, we obtain a stalkwise isomorphism of sheaves of fields on the same Stone space, hence an isomorphism of Pierce sheaves. Therefore, $h$ is an isomorphism in $\RegRng$, and a fortiori in $\RegRng_{\pi}$.
\end{proof}

Since $\RegRng$ has transferable monos \cite[Theorem~1.6]{Cornish1977}, it follows from Corollary~\ref{cor:S-perp-Beth-comp} that $\RegRng_{\pi}$ is the Beth companion of $\RegRng$, and even the strong Beth companion because $\RegRng$ satisfies \AP. We can also deduce from Theorem~\ref{t:transfer-quasi-variety} that $\RegRng_{\pi}$ is equivalent to a (mono-sorted, finitary) variety of algebras. For this, we need to show that the rings $U_{p}$ are regular projective in $\RegRng$. That is, for every surjection $h\colon A\to B$ in $\RegRng$, if $b\in B$ satisfies $pb=0$, there exists $a\in A$ such that $pa=0$ and $h(a)=b$. In turn, this is a consequence of Lemma~\ref{l:lifting-p-annihilation}: if $a'\in A$ is any element satisfying $h(a')=b$, then $a\coloneqq (1-e)a'$ satisfies $pa=0$ and
\[
h(a)=(1-e)h(a')=(1-e)b=b - pp^{*}b = b-0 =b.
\]
For an explicit equational axiomatisation of a variety equivalent to $\RegRng_{\pi}$, see \cite[Theorem~3.14]{CKM2026rings}.  

\subsubsection*{Reduced rings}
Let $\RedRng$ denote the full subcategory of $\CRng$ consisting of reduced rings, i.e.\ those satisfying $x^{2}=0 \, \IMP \, x=0$. Reduced rings form a sub-quasivariety of rings, hence an lfp category. We claim that regular rings are a mono-reflective subcategory of reduced rings and the inclusion $\RegRng\into \RedRng$ is an lfp morphism. Since lfp morphisms compose, and so do mono-reflective subcategories, it will follow from the previous example that $\RegRng_{\pi}$ is the strong Beth companion of $\RedRng$.

The inclusion $\RegRng\into \CRng$ preserves limits and directed colimits. This follows, e.g., by observing that $\RegRng$ is the $\omega$-orthogonality class in $\CRng$ induced by the single ring homomorphism
\[
\Z[x] \to \Z[x,y]/\langle x^{2}y-x, y^{2}x-y\rangle
\]
that sends the generator $x$ to itself.\footnote{An equivalent description of regular rings as an orthogonality class in rings can be found in~\cite[Proposition~4.41]{BSY}.}
Moreover, the inclusion $\RedRng\into \CRng$ reflects limits and colimits because it is fully faithful. Thus, the inclusion $\RegRng\into \RedRng$ preserves limits and directed colimits, i.e., it is an lfp morphism. To see that $\RegRng$ is mono-reflective in $\RedRng$, it suffices to note that every reduced ring can be embedded in a regular one. This is a well-known fact: just observe that every reduced ring embeds in a Cartesian product of fields (consider the product of the fields of fractions of the integral domains obtained by quotienting at prime ideals, and recall that the intersection of all prime ideals in a reduced ring is trivial as it coincides with the nilradical), and any Cartesian product of fields is a regular ring.

\subsubsection*{Reduced $k$-algebras} 

Fix a field $k$, viewed as an object of $\CRng$. The category of \emph{(commutative) $k$-algebras} can be identified with the coslice category $k/\CRng$. Similarly, reduced and regular $k$-algebras can be identified with the coslices $k/\RedRng$ and $k/\RegRng$, respectively. We claim that, if the field~$k$ has characteristic~$0$, then regular $k$-algebras are the strong Beth companion of reduced $k$-algebras.

\begin{lemma}\label{l:monoreflection-coslices}
Let $\cat$ be a mono-reflective subcategory of $\catt$ in the category of lfp categories. The following statements hold:
\begin{enumerate}[label=(\roman*)]
\item\label{i:coslice-adj} for any object $c$ of $\cat$, the coslice $c/\cat$ is a mono-reflective subcategory of $c/\catt$ in the category of lfp categories;
\item\label{i:coslice-TM} if $\cat$ has transferable monos, then so does $c/\cat$.
\end{enumerate}
\end{lemma}

\begin{proof}
Recall that, for any lfp category $\sf E$ and any object $e$ of $\sf E$, the coslice $e/{\sf E}$ is also lfp \cite[Proposition~1.57]{ar94book}. Moreover, the forgetful functor $e/{\sf E}\to {\sf E}$ creates limits and connected colimits (see e.g.\ \cite[Proposition~3.3.8]{Riehl2016book}). In particular, it creates monomorphisms, directed colimits and pushouts.

\ref{i:coslice-adj} Denote by $\eta$ the unit of the adjunction formed by the inclusion $\cat \into \catt$ and its left adjoint. The inclusion $c/\cat \into c/\catt$ admits a left adjoint that sends an object $f\colon c\to A$ to the composite $\eta_{A} \circ f$. The unit at~$f$ of the latter reflection is $\eta_{A}$, regarded as an arrow in $c/\catt$. Since $\eta_{A}$ is monic in $\catt$ and the forgetful functor $c/\catt\to \catt$ creates monomorphisms, it follows that $c/\cat$ is mono-reflective in $c/\catt$. Moreover, the inclusion $c/\cat\into c/\catt$ preserves directed colimits because $\cat\into \catt$ does, and directed colimits are a particular type of connected colimits.

\ref{i:coslice-TM} Suppose that $\cat$ has transferable monos. Equivalently, monomorphisms in~$\cat$ are stable under all pushouts. Because the forgetful functor $c/\cat\to \cat$ creates monomorphisms and pushouts, we conclude that $c/\cat$ has transferable monos.
\end{proof}

As mentioned in the preceding examples, $\RegRng$ has transferable monos and is a mono-reflective subcategory of $\RedRng$ in the category of lfp categories.
Therefore, the previous lemma entails that $k/\RegRng$ has transferable monos and is a full mono-reflective subcategory of $k/\RedRng$ in the category of lfp categories. Regarding balancedness, we note the following fact.

\begin{proposition}
If $k$ has characteristic $0$, then $k/\RegRng$ is balanced.
\end{proposition}

\begin{proof}
Observe that $k/\RegRng$ is a reflective subcategory of $k/\CRng$, although not a mono-reflective one. This follows from the proof of Lemma~\ref{l:monoreflection-coslices}, recalling that $\RegRng$ is a reflective subcategory of $\CRng$ (and even an $\omega$-orthogonality class in~$\CRng$). In particular, the inclusion $k/\RegRng\into k/\CRng$ preserves monomorphisms. If we show that it sends epis to surjections, it will follow that it sends epi-monos to isomorphisms, and therefore $k/\RegRng$ is balanced.

Let $f\colon A \to B$ be an epimorphism in $k/\RegRng$. The cokernel pair of $f$ is given by taking the pushout of $f$ along itself in $k/\CRng$, which yields the tensor product $R\coloneqq B\otimes_{A} B$, and then applying the left adjoint to the inclusion $k/\RegRng \into k/\CRng$. Write $\nu(R)$ for the resulting regular $k$-algebra, and let $\eta\colon R \to \nu(R)$ be the reflection morphism. The kernel of $\eta$ is the nilradical of $R$ (as can be seen by factoring $k/\RegRng \into k/\CRng$ via the mono-reflective subcategory $k/\RegRng \into k/\RedRng$). 

We claim that $R$ is reduced, and so $\eta$ is injective. It is enough to show that $R$ embeds in a product of reduced rings. Recall that the canonical map from any $A$-module to the product of its localisations at the maximal ideals of $A$ is injective. In particular, there is an injective homomorphism
\[
R \into \prod_{m\in \mathrm{Max}(A)}{R_{m}}
\] 
where $\mathrm{Max}(A)$ is the set of maximal ideals of $A$, and $R_{m}$ is the localisation of $R$ at~$A\setminus m$. It suffices to prove that each $R_{m}$ is reduced. Localisations commute with tensor products, hence we get
\[
R_{m} = (B\otimes_{A} B)_{m} \cong B_{m} \otimes_{A_{m}} B_{m}.
\]
Since $A$ is regular, the localisation $A_{m}$ is a field (see e.g.\ \cite[Theorem~1.16]{Goodearl1991}), and so it is a field extension of $k$. Because $k$ has characteristic $0$, so does $A_{m}$; in particular, $A_{m}$ is a perfect field.
Moreover, regular rings are closed under localisations. Just observe that, if $C$ is a regular ring and $S\subseteq C$ is a multiplicative subset, then for each element $x=\frac{a}{s}\in S^{-1}C$, there exists $b\in C$ such that $a=a^{2}b$. The element $y\coloneqq sb$ then satisfies
\[
x = \frac{a}{s} = \frac{a^{2}b}{s} = \bigg(\frac{a}{s}\bigg)^{2} (sb) = x^{2} y,
\]
showing that $S^{-1}C$ is regular. Therefore, $B_{m}$ is a regular ring and, in particular, reduced. For any perfect field $K$, the tensor product of reduced $K$-algebras is reduced (see e.g.\ \cite[Ch.~V, \S 15, Th\'eor\`eme~3]{BourbakiAlgebra4-7}). Therefore, $R_{m}$ is reduced. 

Now, consider the cokernel pair of $f$ in $k/\CRng$, as displayed below.
\[\begin{tikzcd}
B \arrow[yshift=3pt]{r}{c_{0}} \arrow[yshift=-3pt]{r}[swap]{c_{1}} & R
\end{tikzcd}\]
The cokernel pair of $f$ in $k/\RegRng$ is $(\eta\circ c_{0},\eta\circ c_{1})$. As $f$ is epic in $k/\RegRng$, we get $\eta\circ c_{0} = \eta\circ c_{1}$, and thus $c_{0}=c_{1}$ because $\eta$ is injective. That is, $f$ is an epimorphism in $k/\CRng$. Viewing epimorphisms as connected colimits, it follows that~$f$ is an epimorphism in $\CRng$. In turn, every epimorphism in $\CRng$ whose domain is regular is a surjection; see e.g.\ \cite[Theorems~2.1 and~3.12]{Tarizadeh2022}.
\end{proof}

Recall from Section~\ref{s:preliminaries-amalgamation} that, in the presence of the amalgamation property (and, a fortiori, in any category with transferable monos), extremal monomorphisms are regular monomorphisms, and so properties \ES~and \SES~are equivalent. Therefore, if~$k$ has characteristic~$0$, $k/\RegRng$ is the strong Beth companion of $k/\RedRng$.

\appendix
\section{Acronyms}\label{s:acronyms}

In this short appendix, we list the balancedness- and amalgamation-type properties introduced in the paper. 

\begin{table}[H]
\centering
\begin{tabular}{ |c|c|c| } 
 \hline
 \ES & ``Epimorphisms are surjective'' & \multirow{2}{*}{Definition~\ref{d:ES-and-SES}} \\ 
 \cline{1-2}
 \SES & Strong \ES~property &  \\ 
 \hline
\end{tabular}
\caption{Balancedness-type properties}
\end{table}

\begin{table}[H]
\centering
\begin{tabular}{ |c|c|c| } 
 \hline
 \AP & Amalgamation property & \multirow{3}{*}{Definition~\ref{d:AP-SAP-IPA}} \\ 
  \cline{1-2}
 \SAP & Strong amalgamation property &  \\ 
  \cline{1-2}
 \IPA & Intersection property of amalgamations &  \\ 
 \hline
 \EAP & Epimorphism amalgamation property & Definition~\ref{d:EAP} \\
 \hline
 \TEM & Transferable epi-monos & Definition~\ref{d:TEM} \\
 \hline
 \TM & Transferable monos & Remark~\ref{r:TM} \\
 \hline
\end{tabular}
\caption{Amalgamation-type properties}
\end{table}

\bibliographystyle{alphaurl-shortnames}

\begin{thebibliography}{DLLNS21}

\bibitem[ADV25]{ADV2025}
J.~Ad{\'a}mek, M.~Dost{\'a}l, and J.~Velebil.
\newblock Sifted colimits, strongly finitary monads and continuous algebras.
\newblock {\em Theory Appl. Categ.}, 44(3):84--131, 2025.
\newblock \href {https://doi.org/10.70930/tac/qa04oz7o}
  {\path{doi:10.70930/tac/qa04oz7o}}.

\bibitem[AHS06]{ahs06book}
J.~Ad{\'a}mek, H.~Herrlich, and G.~Strecker.
\newblock {\em {Abstract and concrete categories: The joy of cats}}, volume~17.
\newblock Reprints in Theory and Applications of Categories, 2006.
\newblock Originally published by John Wiley and Sons, New York, 1990.

\bibitem[AHS09]{AHS2009}
J.~Ad{\'a}mek, M.~H{\'e}bert, and L.~Sousa.
\newblock The orthogonal subcategory problem and the small object argument.
\newblock {\em Applied Categorical Structures}, 17(3):211--246, 2009.
\newblock \href {https://doi.org/10.1007/s10485-008-9153-4}
  {\path{doi:10.1007/s10485-008-9153-4}}.

\bibitem[ALN25]{ALN25}
B.~Ahrens, P.~L. Lumsdaine, and P.~R. North.
\newblock Comparing semantic frameworks for dependently-sorted algebraic
  theories.
\newblock In O.~Kiselyov, editor, {\em Programming Languages and Systems},
  pages 3--22, Singapore, 2025. Springer Nature Singapore.
\newblock \href {https://doi.org/10.1007/978-981-97-8943-6_1}
  {\path{doi:10.1007/978-981-97-8943-6_1}}.

\bibitem[AP98]{AP1998}
J.~Ad{\'a}mek and H.-E. Porst.
\newblock Algebraic theories of quasivarieties.
\newblock {\em J. Algebra}, 208(2):379--398, 1998.
\newblock \href {https://doi.org/10.1006/jabr.1998.7499}
  {\path{doi:10.1006/jabr.1998.7499}}.

\bibitem[AR94]{ar94book}
J.~Ad{\'a}mek and J.~Rosick{\'y}.
\newblock {\em {Locally Presentable and Accessible Categories}}, volume 189 of
  {\em London Mathematical Society Lecture Notes Series}.
\newblock Cambridge University Press, 1994.
\newblock \href {https://doi.org/10.1017/CBO9780511600579}
  {\path{doi:10.1017/CBO9780511600579}}.

\bibitem[AS03]{AS2003}
J.~Ad{\'a}mek and L.~Sousa.
\newblock How accessible are categories of algebras?
\newblock {\em J. Pure Appl. Algebra}, 182(1):1--15, 2003.
\newblock \href {https://doi.org/10.1016/S0022-4049(02)00328-6}
  {\path{doi:10.1016/S0022-4049(02)00328-6}}.

\bibitem[AS04]{AS2004}
J.~Ad{\'a}mek and L.~Sousa.
\newblock On reflective subcategories of varieties.
\newblock {\em Journal of Algebra}, 276(2):685--705, 2004.
\newblock \href {https://doi.org/10.1016/j.jalgebra.2003.09.039}
  {\path{doi:10.1016/j.jalgebra.2003.09.039}}.

\bibitem[Bar89]{Barr1989}
M.~Barr.
\newblock Models of {H}orn theories.
\newblock In {\em Categories in computer science and logic ({B}oulder, {CO},
  1987)}, volume~92 of {\em Contemp. Math.}, pages 1--7. Amer. Math. Soc.,
  Providence, RI, 1989.
\newblock \href {https://doi.org/10.1090/conm/092/1003191}
  {\path{doi:10.1090/conm/092/1003191}}.

\bibitem[BH06]{BH2006}
W.~J. Blok and E.~Hoogland.
\newblock The {B}eth property in algebraic logic.
\newblock {\em Studia Logica}, 83(1-3):49--90, 2006.
\newblock \href {https://doi.org/10.1007/s11225-006-8298-0}
  {\path{doi:10.1007/s11225-006-8298-0}}.

\bibitem[Bir84]{Bird1984}
G.~J. Bird.
\newblock {\em Limits in 2-Categories of Locally-Presentable Categories}.
\newblock PhD thesis, University of Sydney, 1984.
\newblock Circulated by the Sydney Category Theory Seminar.

\bibitem[BK72]{BK1972}
A.~K. Bousfield and D.~M. Kan.
\newblock The core of a ring.
\newblock {\em J. Pure Appl. Algebra}, 2:73--81, 1972.
\newblock \href {https://doi.org/10.1016/0022-4049(72)90023-0}
  {\path{doi:10.1016/0022-4049(72)90023-0}}.

\bibitem[BKR15]{BKR2015}
M.~Barr, J.~F. Kennison, and R.~Raphael.
\newblock Limit closures of classes of commutative rings.
\newblock {\em Theory Appl. Categ.}, 30:Paper No. 8, 229--304, 2015.
\newblock \href {https://doi.org/10.70930/tac/cjcn2avw}
  {\path{doi:10.70930/tac/cjcn2avw}}.

\bibitem[Bou81]{BourbakiAlgebra4-7}
N.~Bourbaki.
\newblock {\em \'El\'ements de math\'ematique. Alg\`ebre. Chapitres 4 \`a{} 7.}
\newblock Masson, Paris, 1981.

\bibitem[Bou10]{Bourke2010Codescent}
J.~Bourke.
\newblock {\em Codescent Objects in 2-Dimensional Universal Algebra}.
\newblock {PhD} thesis, The University of Sydney, 2010.
\newblock URL: \url{https://www.math.muni.cz/~bourkej/papers/JohnBThesis.pdf}.

\bibitem[BSY]{BSY}
R.~Burklund, T.~M. Schlank, and A.~Yuan.
\newblock The chromatic {N}ullstellensatz.
\newblock {\em Annals of Mathematics}.
\newblock To appear.
\newblock URL: \url{https://annals.math.princeton.edu/articles/22772}.

\bibitem[Car86]{Cartmell1986GeneralisedAlgebraicTheories}
J.~Cartmell.
\newblock Generalised algebraic theories and contextual categories.
\newblock {\em Annals of Pure and Applied Logic}, 32:209--243, 1986.
\newblock \href {https://doi.org/10.1016/0168-0072(86)90053-9}
  {\path{doi:10.1016/0168-0072(86)90053-9}}.

\bibitem[Che76]{Cherlin1976}
G.~Cherlin.
\newblock {\em Model theoretic algebra---selected topics}, volume 521 of {\em
  Lecture Notes in Mathematics}.
\newblock Springer-Verlag, Berlin-New York, 1976.
\newblock \href {https://doi.org/10.1007/BFb0079565}
  {\path{doi:10.1007/BFb0079565}}.

\bibitem[CHK85]{CHK1985}
C.~Cassidy, M.~H\'ebert, and G.~M. Kelly.
\newblock Reflective subcategories, localizations and factorization systems.
\newblock {\em J. Austral. Math. Soc. Ser. A}, 38(3):287--329, 1985.
\newblock \href {https://doi.org/10.1017/S1446788700023624}
  {\path{doi:10.1017/S1446788700023624}}.

\bibitem[CKM26a]{CKM2026bis}
L.~Carai, M.~Kurtzhals, and T.~Moraschini.
\newblock A categorical description of simple {B}eth companions, 2026.
\newblock \href {https://arxiv.org/abs/2605.09141} {\path{arXiv:2605.09141}}.

\bibitem[CKM26b]{CKM2026rings}
L.~Carai, M.~Kurtzhals, and T.~Moraschini.
\newblock A completion of reduced commutative rings, 2026.
\newblock \href {https://arxiv.org/abs/2605.12661} {\path{arXiv:2605.12661}}.

\bibitem[CKM26c]{CKM2026}
L.~Carai, M.~Kurtzhals, and T.~Moraschini.
\newblock The theory of implicit operations, 2026.
\newblock \href {https://arxiv.org/abs/2512.14326} {\path{arXiv:2512.14326}}.

\bibitem[Cor77]{Cornish1977}
W.~H. Cornish.
\newblock Amalgamating commutative regular rings.
\newblock {\em Comment. Math. Univ. Carolinae}, 18(3):423--436, 1977.

\bibitem[DLLNS21]{DiLibertiEtAl2021PartialTheories}
I.~Di~Liberti, F.~Loregian, C.~Nester, and P.~Soboci{\'n}ski.
\newblock Functorial semantics for partial theories.
\newblock {\em Proceedings of the ACM on Programming Languages},
  5(POPL):57:1--57:28, 2021.
\newblock \href {https://doi.org/10.1145/3434338} {\path{doi:10.1145/3434338}}.

\bibitem[FK72]{FK1972}
P.~J. Freyd and G.~M. Kelly.
\newblock Categories of continuous functors, {I}.
\newblock {\em Journal of Pure and Applied Algebra}, 2(3):169--191, 1972.
\newblock \href {https://doi.org/10.1016/0022-4049(72)90001-1}
  {\path{doi:10.1016/0022-4049(72)90001-1}}.

\bibitem[Fre72]{Freyd1972AspectsOfTopoi}
P.~J. Freyd.
\newblock Aspects of topoi.
\newblock {\em Bulletin of the Australian Mathematical Society}, 7(1):1--76,
  1972.
\newblock \href {https://doi.org/10.1017/S0004972700044828}
  {\path{doi:10.1017/S0004972700044828}}.

\bibitem[Fre02]{Freyd2002CartesianLogic}
P.~J. Freyd.
\newblock Cartesian logic.
\newblock {\em Theoretical Computer Science}, 278(1--2):3--21, May 2002.
\newblock \href {https://doi.org/10.1016/S0304-3975(00)00328-5}
  {\path{doi:10.1016/S0304-3975(00)00328-5}}.

\bibitem[Fre25]{Frey2025DualityClans}
J.~Frey.
\newblock Duality for clans: An extension of {Gabriel--Ulmer} duality.
\newblock {\em The Journal of Symbolic Logic}, pages 1--38, 2025.
\newblock \href {https://doi.org/10.1017/jsl.2024.79}
  {\path{doi:10.1017/jsl.2024.79}}.

\bibitem[Goo91]{Goodearl1991}
K.~R. Goodearl.
\newblock {\em Von {N}eumann regular rings}.
\newblock Robert E. Krieger Publishing Co., Inc., Malabar, FL, second edition,
  1991.

\bibitem[GU71]{GabrielUlmer1971}
P.~Gabriel and F.~Ulmer.
\newblock {\em Lokal pr{\"a}sentierbare Kategorien}, volume 221 of {\em Lecture
  Notes in Mathematics}.
\newblock Springer-Verlag, Berlin and Heidelberg, 1971.
\newblock \href {https://doi.org/10.1007/BFb0059396}
  {\path{doi:10.1007/BFb0059396}}.

\bibitem[HAR01a]{HAR2001}
M.~H{\'e}bert, J.~Adamek, and J.~Rosick\'y.
\newblock More on orthogonality in locally presentable categories.
\newblock {\em Cahiers de Topologie et G\'eom\'etrie Diff\'erentielle
  Cat\'egoriques}, 42(1):51--80, 2001.

\bibitem[HAR01b]{HebertAdamekRosicky2001}
M.~H{\'e}bert, J.~Ad{\'a}mek, and J.~Rosick{\'y}.
\newblock More on orthogonality in locally presentable categories.
\newblock {\em Cahiers de Topologie et G{\'e}om{\'e}trie Diff{\'e}rentielle
  Cat{\'e}goriques}, 42(1):51--80, 2001.
\newblock URL: \url{https://www.numdam.org/item/CTGDC_2001__42_1_51_0/}.

\bibitem[H{\'e}b93]{Hebert1993}
M.~H{\'e}bert.
\newblock Sur les op\'erations partielles implicites et leur relation avec la
  surjectivit\'e{} des \'epimorphismes.
\newblock {\em Canad. J. Math.}, 45(3):554--575, 1993.
\newblock \href {https://doi.org/10.4153/CJM-1993-029-3}
  {\path{doi:10.4153/CJM-1993-029-3}}.

\bibitem[H{\'e}b98]{Hebert1998}
M.~H{\'e}bert.
\newblock On generation and implicit partial operations in locally presentable
  categories.
\newblock {\em Appl. Categ. Structures}, 6(4):473--488, 1998.
\newblock \href {https://doi.org/10.1023/A:1008653513618}
  {\path{doi:10.1023/A:1008653513618}}.

\bibitem[HI67]{HI1967}
J.~M. Howie and J.~R. Isbell.
\newblock Epimorphisms and dominions. {II}.
\newblock {\em J. Algebra}, 6:7--21, 1967.
\newblock \href {https://doi.org/10.1016/0021-8693(67)90010-5}
  {\path{doi:10.1016/0021-8693(67)90010-5}}.

\bibitem[HMT85]{HMT1985}
L.~Henkin, J.~D. Monk, and A.~Tarski.
\newblock {\em Cylindric algebras. {P}art {II}}, volume 115 of {\em Studies in
  Logic and the Foundations of Mathematics}.
\newblock North-Holland Publishing Co., Amsterdam, 1985.

\bibitem[Hoo00]{Hoogland2000}
E.~Hoogland.
\newblock Algebraic characterizations of various {B}eth definability
  properties.
\newblock {\em Studia Logica}, 65(1):91--112, 2000.
\newblock \href {https://doi.org/10.1023/A:1005295109904}
  {\path{doi:10.1023/A:1005295109904}}.

\bibitem[Isb66]{Isbell1966}
J.~R. Isbell.
\newblock Epimorphisms and dominions.
\newblock In {\em Proc. {C}onf. {C}ategorical {A}lgebra ({L}a {J}olla,
  {C}alif., 1965)}, pages 232--246. Springer-Verlag New York, Inc., New York,
  1966.
\newblock \href {https://doi.org/10.1007/978-3-642-99902-4_9}
  {\path{doi:10.1007/978-3-642-99902-4_9}}.

\bibitem[Joh82]{Johnstone1982}
P.~T. Johnstone.
\newblock {\em Stone spaces}, volume~3 of {\em Cambridge Studies in Advanced
  Mathematics}.
\newblock Cambridge University Press, Cambridge, 1982.

\bibitem[Kel89]{Kelly1989}
G.~M. Kelly.
\newblock Elementary observations on 2-categorical limits.
\newblock {\em Bulletin of the Australian Mathematical Society},
  39(2):301--317, 1989.
\newblock \href {https://doi.org/10.1017/S0004972700002781}
  {\path{doi:10.1017/S0004972700002781}}.

\bibitem[KL93]{KL1993}
G.~M. Kelly and S.~Lack.
\newblock Finite-product-preserving functors, {K}an extensions and
  strongly-finitary {$2$}-monads.
\newblock {\em Appl. Categ. Structures}, 1(1):85--94, 1993.
\newblock \href {https://doi.org/10.1007/BF00872987}
  {\path{doi:10.1007/BF00872987}}.

\bibitem[KMPT83]{KMPT1983}
E.~W. Kiss, L.~M{\'a}rki, P.~Pr{\"o}hle, and W.~Tholen.
\newblock Categorical algebraic properties. {A} compendium of amalgamation,
  congruence extension, epimorphisms, residual smallness, and injectivity.
\newblock {\em Stud. Sci. Math. Hung.}, 18:79--141, 1983.

\bibitem[Lac10]{Lack2010}
S.~Lack.
\newblock A 2-categories companion.
\newblock In {\em Towards higher categories}, volume 152 of {\em IMA Vol. Math.
  Appl.}, pages 105--191. Springer, New York, 2010.
\newblock \href {https://doi.org/10.1007/978-1-4419-1524-5\_4}
  {\path{doi:10.1007/978-1-4419-1524-5\_4}}.

\bibitem[Law68]{Lawvere1968}
F.~W. Lawvere.
\newblock Some algebraic problems in the context of functorial semantics of
  algebraic theories.
\newblock In {\em Reports of the {M}idwest {C}ategory {S}eminar, {II}}, volume
  No. 61 of {\em Lecture Notes in Math.}, pages 41--61. Springer, Berlin-New
  York, 1968.
\newblock \href {https://doi.org/10.1007/BFb0077116}
  {\path{doi:10.1007/BFb0077116}}.

\bibitem[LP09]{LackPower2009}
S.~Lack and J.~Power.
\newblock {Gabriel--Ulmer} duality and {Lawvere} theories enriched over a
  general base.
\newblock {\em Journal of Functional Programming}, 19(3--4):265--286, 2009.
\newblock \href {https://doi.org/10.1017/S0956796809007254}
  {\path{doi:10.1017/S0956796809007254}}.

\bibitem[LS73]{LS1973}
L.~Lipshitz and D.~Saracino.
\newblock The model companion of the theory of commutative rings without
  nilpotent elements.
\newblock {\em Proc. Amer. Math. Soc.}, 38:381--387, 1973.
\newblock \href {https://doi.org/10.2307/2039295} {\path{doi:10.2307/2039295}}.

\bibitem[Mad67]{Maddox1967}
B.~H. Maddox.
\newblock Absolutely pure modules.
\newblock {\em Proc. Amer. Math. Soc.}, 18:155--158, 1967.
\newblock \href {https://doi.org/10.2307/2035245} {\path{doi:10.2307/2035245}}.

\bibitem[Met26]{Metcalfe2026}
G.~Metcalfe.
\newblock Interpolation and amalgamation.
\newblock In B.~ten Cate, J.~C. Jung, P.~Koopmann, C.~Wernhard, and F.~Wolter,
  editors, {\em Theory and Applications of Craig Interpolation}. Ubiquity
  Press, 2026.
\newblock \href {https://doi.org/10.5334/bdg} {\path{doi:10.5334/bdg}}.

\bibitem[MP87a]{MP1987}
M.~Makkai and A.~M. Pitts.
\newblock Some results on locally finitely presentable categories.
\newblock {\em Trans. Amer. Math. Soc.}, 299(2):473--496, 1987.
\newblock \href {https://doi.org/10.2307/2000508} {\path{doi:10.2307/2000508}}.

\bibitem[MP87b]{MakkaiPitts1987}
M.~Makkai and A.~M. Pitts.
\newblock Some results on locally finitely presentable categories.
\newblock {\em Transactions of the American Mathematical Society},
  299(2):473--496, 1987.
\newblock \href {https://doi.org/10.1090/S0002-9947-1987-0869216-2}
  {\path{doi:10.1090/S0002-9947-1987-0869216-2}}.

\bibitem[Neu36]{vonNeumann1936}
J.~V. Neumann.
\newblock On regular rings.
\newblock {\em Proceedings of the National Academy of Sciences of the United
  States of America}, 22(12):707--713, 1936.
\newblock \href {https://doi.org/10.1073/pnas.22.12.707}
  {\path{doi:10.1073/pnas.22.12.707}}.

\bibitem[Pie67]{Pierce1967}
R.~S. Pierce.
\newblock {\em Modules over commutative regular rings}, volume No. 70 of {\em
  Memoirs of the American Mathematical Society}.
\newblock American Mathematical Society, Providence, RI, 1967.
\newblock \href {https://doi.org/10.1090/memo/0070}
  {\path{doi:10.1090/memo/0070}}.

\bibitem[PV07]{PalmgrenVickers2007PartialHorn}
E.~Palmgren and S.~J. Vickers.
\newblock Partial {Horn} logic and {Cartesian} categories.
\newblock {\em Annals of Pure and Applied Logic}, 145(3):314--353, March 2007.
\newblock \href {https://doi.org/10.1016/j.apal.2006.10.001}
  {\path{doi:10.1016/j.apal.2006.10.001}}.

\bibitem[RAB02]{RAB2022}
J.~Rosick\'y, J.~Ad\'amek, and F.~Borceux.
\newblock More on injectivity in locally presentable categories.
\newblock {\em Theory Appl. Categ.}, 10:No. 7, 148--161, 2002.
\newblock \href {https://doi.org/10.70930/tac/is5v1sx1}
  {\path{doi:10.70930/tac/is5v1sx1}}.

\bibitem[Reg26]{Reggio2026}
L.~Reggio.
\newblock On implicit operations and balanced categories, 2026.
\newblock \href {https://arxiv.org/abs/2609.15626} {\path{arXiv:2609.15626}}.

\bibitem[Rei70]{Reid1969}
G.~A. Reid.
\newblock Epimorphisms and surjectivity.
\newblock {\em Invent. Math.}, 9:295--307, 1969/70.
\newblock \href {https://doi.org/10.1007/BF01425484}
  {\path{doi:10.1007/BF01425484}}.

\bibitem[Rie14]{Riehl2014}
E.~Riehl.
\newblock {\em Categorical homotopy theory}, volume~24 of {\em New Mathematical
  Monographs}.
\newblock Cambridge University Press, Cambridge, 2014.
\newblock \href {https://doi.org/10.1017/CBO9781107261457}
  {\path{doi:10.1017/CBO9781107261457}}.

\bibitem[Rie16]{Riehl2016book}
E.~Riehl.
\newblock {\em Category theory in context}.
\newblock Aurora Dover Modern Math Originals. Dover Publications, Inc.,
  Mineola, NY, 2016.

\bibitem[Rin71]{Ringel1971}
C.~M. Ringel.
\newblock Monofunctors as reflectors.
\newblock {\em Trans. Amer. Math. Soc.}, 161:293--306, 1971.
\newblock \href {https://doi.org/10.2307/1995944} {\path{doi:10.2307/1995944}}.

\bibitem[Rob77]{Rob77}
A.~Robinson.
\newblock {\em Complete theories}.
\newblock North-Holland Publishing Co., Amsterdam-New York-Oxford, second
  edition, 1977.
\newblock With a preface by H.J. Keisler, Studies in Logic and the Foundations
  of Mathematics.

\bibitem[Rot97]{Rothmaler1997}
P.~Rothmaler.
\newblock Purity in model theory.
\newblock In {\em Advances in algebra and model theory ({E}ssen, 1994;
  {D}resden, 1995)}, volume~9 of {\em Algebra Logic Appl.}, pages 445--469.
  Gordon and Breach, Amsterdam, 1997.
\newblock \href {https://doi.org/10.1201/9780367810603-24}
  {\path{doi:10.1201/9780367810603-24}}.

\bibitem[SAC68]{SAC1968}
{\em S\'eminaire d'{A}lg\`ebre {C}ommutative dirig\'e{} par {P}ierre {S}amuel:
  1967/1968. {L}es \'epimorphismes d'anneaux}.
\newblock Secr\'etariat math\'ematique, Paris, 1968.

\bibitem[Sto68]{Storrer1968}
H.~H. Storrer.
\newblock Epimorphismen von kommutativen {R}ingen.
\newblock {\em Comment. Math. Helv.}, 43:378--401, 1968.
\newblock \href {https://doi.org/10.1007/BF02564404}
  {\path{doi:10.1007/BF02564404}}.

\bibitem[Szp30]{Szpilrajn1930}
E.~Szpilrajn.
\newblock Sur l'extension de l'ordre partiel.
\newblock {\em Fundamenta Mathematicae}, 16(1):386--389, 1930.
\newblock \href {https://doi.org/10.4064/fm-16-1-386-389}
  {\path{doi:10.4064/fm-16-1-386-389}}.

\bibitem[Tar22]{Tarizadeh2022}
A.~Tarizadeh.
\newblock On flat epimorphisms of rings and pointwise localizations.
\newblock {\em Mathematica}, 64(87)(1):129--138, 2022.
\newblock \href {https://doi.org/10.24193/mathcluj.2022.1.14}
  {\path{doi:10.24193/mathcluj.2022.1.14}}.

\bibitem[Ten25]{Tendas2025}
G.~Tendas.
\newblock Dualities in the theory of accessible categories.
\newblock {\em Journal of Algebra}, 674:29--49, 2025.
\newblock \href {https://doi.org/10.1016/j.jalgebra.2025.03.012}
  {\path{doi:10.1016/j.jalgebra.2025.03.012}}.

\bibitem[Tho82]{Tholen1982}
W.~Tholen.
\newblock Amalgamations in categories.
\newblock {\em Algebra Universalis}, 14(3):391--397, 1982.
\newblock \href {https://doi.org/10.1007/BF02483940}
  {\path{doi:10.1007/BF02483940}}.

\bibitem[Vol79]{Volger1979}
H.~Volger.
\newblock Preservation theorems for limits of structures and global sections of
  sheaves of structures.
\newblock {\em Math. Z.}, 166(1):27--54, 1979.
\newblock \href {https://doi.org/10.1007/BF01173845}
  {\path{doi:10.1007/BF01173845}}.

\end{thebibliography}

\end{document}